\documentclass{birkjour}
 \newtheorem{thm}{Theorem}[section]
 
 \newtheorem{lem}[thm]{Lemma}
 
 \theoremstyle{definition}
 \newtheorem{defn}[thm]{Definition}
 \theoremstyle{remark}
 \newtheorem{rem}[thm]{Remark}
 
 \numberwithin{equation}{section}
 \usepackage{cases}
 \usepackage[colorlinks=true,
 linkcolor=blue, 
 anchorcolor=black, 
 citecolor=blue, 
 urlcolor=blue
 ]{hyperref}
 \usepackage{cite}
 \usepackage{subfigure}
 \allowdisplaybreaks
 
\begin{document}

%
%
%
%
%
%
%
%
%

\title[A class of forward-backward diffusion equations for ...]
 {A class of forward-backward diffusion equations for multiplicative noise removal}

\author[Y. Tong]{Yihui Tong}

\address{%
School of Mathematics\\
Harbin Institute of Technology\\
Harbin 150001 Heilongjiang\\
China}

\email{23B912019@stu.hit.edu.cn}

\author[W. Liu]{Wenjie Liu}

\address{%
School of Mathematics\\
Harbin Institute of Technology\\
Harbin 150001 Heilongjiang\\
China}

\email{liuwenjie@hit.edu.cn}

\author[Z. Guo]{Zhichang Guo}

\address{%
School of Mathematics\\
Harbin Institute of Technology\\
Harbin 150001 Heilongjiang\\
China}

\email{mathgzc@hit.edu.cn}

\author[W. Yao]{Wenjuan Yao}

\address{%
School of Mathematics\\
Harbin Institute of Technology\\
Harbin 150001 Heilongjiang\\
China}

\email{mathywj@hit.edu.cn}



\date{January 1, 2004}

\begin{abstract}
This paper investigates a class of degenerate forward-backward diffusion equations with a nonlinear source term, proposed as a model for removing multiplicative noise in images. Based on Rothe’s method, the relaxation theorem, and Schauder’s fixed-point theorem, we establish the existence of Young measure solutions for the corresponding initial boundary problem. The continuous dependence result relies on the independence property satisfied by the Young measure solution. Numerical experiments illustrate the denoising effectiveness of our model compared to other denoising models.
\end{abstract}

\maketitle
\section{Introduction}
Multiplicative noise, often appearing in synthetic aperture radar, ultrasound, and laser images, significantly degrades image edges and details \cite{tur1982,shao2020}. Over the past three decades, various methods based on partial differential equations (PDEs) have been developed for image denoising. Given an original image $u$, we assume it has been corrupted by multiplicative noise $\eta$. Our aim is to recover $u$ from $f$ that satisfies $f=u\cdot\eta$. Let $\Omega\subset\mathbb{R}^n$ ($n\geq 2$) be a bounded domain with appropriately smooth boundary $\partial\Omega$ and $0<T<\infty$. We consider the following nonlinear degenerate diffusion equation of forward-backward type:
\begin{numcases}{}
	\frac{\partial u}{\partial t}=\operatorname{div} \left(\alpha(x)\left(\frac{1}{1+\left|\nabla u\right|^2}+\delta\left|\nabla u\right|^{p-2}\right)\nabla u \right)-\lambda\frac{u-f}{u^2}, & \text{in }$Q_T$,\label{P1}\\
	\langle\nabla u,\vec{n}\rangle=0, &\text{in }$S_T$,\label{P2} \\
	u(x,0)=f(x), & \text{in }$\Omega$,\label{P3}
\end{numcases}
where $Q_T=\Omega\times(0,T)$, $S_T=\partial\Omega\times(0,T)$, with parameters $1<p\leq 2$, $\delta>0$, and $\lambda>0$. The positive coefficient $\alpha\in L^{\infty}(\Omega)$ denotes the gray intensity of the noisy image $f$. The forward-backward diffusion nature of equation (\ref{P1}) makes it a suitable model for simultaneously enhancing image edges and reducing noise. However, this characteristic often precludes the existence of classical solutions, leading to the analysis of problems (\ref{P1})--(\ref{P3}) in the sense of Young measure solutions.

Since the work of \cite{rudin2003mul}, numerous variational models \cite{shi2008,li2010,dong2013convex} have been developed to address multiplicative noise. The standard approach to implementing these models involves deriving their corresponding evolution equations and discretizing them to evaluate their performance on noisy images. In general, a variational model yields the following evolution equation:
$$
\frac{\partial u}{\partial t}=\text{div}\left(a(\left|\nabla u\right|,u)\nabla u\right)-\rho h(f,u),\quad \text{in } Q_T,
$$
where $a(\left|\nabla u\right|,u)$ and $h(f,u)$ are given by the corresponding variational model. However, in practice, the denoising problem can be resolved directly using the diffusion equation rather than the evolution equation derived from the variational model; that is, the problem reduces to designing the proper diffusion coefficient $a(\left|\nabla u\right|,u)$ and the source term $h(f,u)$ to guide the image denoising process \cite{zhou2014}. Based on the maximum a posteriori (MAP) estimation of multiplicative Gamma noise, the source term $h(f,u) = \frac{u-f}{u^2}$ was derived in \cite{aubert2008}, providing an effective approach for multiplicative noise removal. For instance, see \cite{liu2013,gaotl2022}. Additionally, the image gray-level indicator $\alpha$ is designed to be much smaller at low gray levels than at high gray levels \cite{zhang2020}, aligning with the characteristic that higher multiplicative noise intensity causes greater damage to the image. This design serves as an effective tool for simultaneously preserving details and removing noise. For further information, we refer to \cite{dong2013,yaowj2019}. Inspired by the work
of \cite{perona1990}, various diffusion coefficients of the form $a(\left|\nabla u\right|,u)$ have been proposed for image denoising applications \cite{guo2011four,guozc2011,shao2020}. To preserve the image restoration capability of the Perona-Malik (PM) equation \cite{perona1990}, Guidotti et al. \cite{guidotti2013} considered a class of forward-backward diffusion equations to remove additive noise
\begin{equation}
    \frac{\partial u}{\partial t}=\text{div}\left(q(\nabla u)\right)=:\text{div}\left(\frac{\nabla u}{1+\left|\nabla u\right|^2}+\delta\left|\nabla u\right|^{p-2}\nabla u\right),\quad\text{in } Q_T,\label{guidottieq}
\end{equation}
with $q=\nabla \varphi$ satisfying the structure conditions:
\begin{equation}
	\max\{\gamma_1\left|\xi\right|^p-1,0\}\leq \varphi(\xi)\leq \gamma_2\left|\xi\right|^p+1,\quad \left|q(\xi)\right|\leq \gamma_2\left|\xi\right|^{p-1},\quad\xi\in \mathbb{R}^n,\label{structure-condition}
\end{equation}
for some constants $0<\gamma_1\leq\gamma_2$ and parameters $1<p\leq 2$, $\delta>0$. Equation (\ref{guidottieq}) can be interpreted either as a spatial regularization of the PM equation or as a nonlinear scale-space model for image edge enhancement, see \cite{li2024,zhou2014,shao2022}. Even after regularization, the equation remains of forward-backward type, at least for $\delta<\frac{1}{8}$. Equation (\ref{guidottieq}) effectively mitigates the staircasing effect resulting from the discretization of the PM equation, leading to a milder micro-structured ramping characterized by alternating gradients of finite size \cite{guidotti2012backward}. To the best of our knowledge, theoretical investigations into the forward-backward reaction-diffusion equation for addressing multiplicative noise in image processing remain scarce.

For problem (\ref{P1})--(\ref{P3}), there are several papers concerning the existence and uniqueness of 
degenerate forward-backward diffusion equations without a source term, namely $\lambda=0$. In this case, equation (\ref{P1}) can be generalized to the following form:
\begin{equation}
    \frac{\partial u}{\partial t}=\operatorname{div} \left(\varPhi\left(\nabla u\right)\right), \quad \text{in } Q_T,\label{p-growth-equation}
\end{equation}
where the heat flux $\varPhi=\nabla\varPsi$ with some potential $\varPsi\in C^1(\mathbb{R}^n)$, and $\varPhi$ and $\varPsi$ satisfy
the same structure conditions (\ref{structure-condition}) as $q$ and $\varphi$, respectively. 

When $\varPsi$ is not convex, violating the monotonicity condition $(\varPhi(x)-\varPhi(y))\cdot(x-y)\geq 0$ for some $x,y\in\mathbb{R}^n$, equation (\ref{p-growth-equation}) constitutes a degenerate, forward-backward parabolic equation that generally admits no classical strong or distributional solutions. Equations like (\ref{p-growth-equation}) arise in modeling phenomena such as melting or freezing during superheating or supercooling, as well as in phase transitions where unstable and metastable states are permitted \cite{elliott1985,visintin2002}. In the one-dimensional case, the study of equation (\ref{p-growth-equation}) is primarily motivated by the Clausius-Duhem inequality. Based on the main constitutive assumption that $\varPhi$ is piecewise affine, H\"{o}llig \cite{hollig1983} established the existence of infinitely many weak solutions for (\ref{p-growth-equation}) under the homogeneous Neumann boundary
condition. The assumption on $\varPhi$ was relaxed by enhancing the regularity of the initial values in \cite{zhang2006}. However, Lair \cite{lair1985} showed that there exists at most one smooth solution. For the 
first or second initial boundary value problem of (\ref{p-growth-equation}), Slemrod \cite{slemrod1991} studied the asymptotic behavior of measure-valued solutions.

Young measure representation was employed for forward-backward diffusion equations in \cite{kinderlehrer1992}, where Kinderlehrer and Pedregal explored the homogeneous Dirichlet boundary problem (\ref{p-growth-equation}) under the structure condition (\ref{structure-condition}). When $p=2$, they defined a Young measure solution $u\in L^{\infty}(0,T;H_0^1(\Omega))$ with $\frac{\partial u}{\partial t}\in L^2(Q_T)$ as follows:
$$
\frac{\partial u}{\partial t}(x,t)=\text{div}\int_{\mathbb{R}^n}\varPhi(\xi)d\nu_{x,t}(\xi),\quad\text{in}~H^{-1}(Q_T),
$$
$$
\nabla u(x,t)=\int_{\mathbb{R}^n}\xi d\nu_{x,t}(\xi),\quad\text{a.e.}~(x,t)\in Q_T,
$$
and (\ref{P3}) holds in the sense of trace, where $(\nu_{x,t})_{(x,t)\in Q_T}$ denotes a $H^1(Q_T)$-gradient Young measure on $\mathbb{R}^n$. Valuable discussions on Young measures are available in \cite{ball1989,young1969,diPerna1985,evans1990}. Kinderlehrer et al. \cite{kinderlehrer1992} established the existence of Young measure solutions by employing Rothe’s method and variational method, incorporating the relaxation theorem. Within this framework, Demoulini \cite{demoulini1996} further explored the qualitative properties of Young measure solutions obtained in \cite{kinderlehrer1992} and found that the uniqueness result depends on the following independence property:
$$
\int_{\mathbb{R}^n}\varPhi(\xi)\cdot\xi d\nu_{x,t}(\xi)=
\int_{\mathbb{R}^n}\varPhi(\xi) d\nu_{x,t}(\xi)\cdot \int_{\mathbb{R}^n}\xi d\nu_{x,t}(\xi),\quad (x,t)\in Q_T,
$$
namely that the heat flux $\varPhi$ and the gradient $\nabla u$ of the solution be independent with respect to the Young measure $\nu$. The results of Demoulini were subsequently extended from the Dirichlet case to the periodic and Neumann settings in \cite{guidotti2012backward}.
When $1\leq p<2$, Yin and Wang \cite{yin2003} studied the first initial-boundary value problem of (\ref{p-growth-equation}). In the case $1<p<2$, they employed the approach similar to that in \cite{demoulini1996} to prove the existence of Young measure solutions. As for the case $p=1$, the equation (\ref{p-growth-equation}) is singular, and Young measure solutions should be considered in $L^{\infty}(0,T;BV(\Omega))$. The uniqueness of the Young measure solution was also analyzed in \cite{yin2003}.

To our knowledge, only a few papers address forward-backward diffusion equations with a source term. For more details, interested readers may refer to \cite{wang2014young}. In
this paper, we study the existence and continuous dependence of Young measure solutions for problem (\ref{P1})--(\ref{P3}). There are three main difficulties to overcome. The first difficulty arises from the presence of the source term, which prevents the direct application of the relaxation theorem to obtain approximate solutions for the Young measure solution of problem (\ref{P1})--(\ref{P3}). As a result, the existence of Young measure solutions cannot be established using the method from \cite{guidotti2012backward}. To address this, we employ Schauder's fixed-point theorem in the proof to establish existence. The second one stems from the nonconvexity of the potential $\varphi$, which implies the minimization of a nonconvex variational principle. To overcome this, we consider the relaxation of the nonconvex functional. The last one is the complexity introduced by the nonlinear dependence of the heat flow $q$ on the gradient of the solution $\nabla u$. This is managed by replacing the nonlinearity $q(\nabla u)$ with the expected value $\int_{\mathbb{R}^n}q(\xi)d\nu(\xi)$, where $q$ is evaluated against a Young measure $\nu$. Our main result is stated as follows:
\begin{thm}\label{thirdtheorem}
    If $f\in W^{1,p}(\Omega)\cap L^{\infty}(\Omega)$ and $\inf\{f(x)\}>0$, then there exists a unique Young measure solution $u$ to the problem \emph{(\ref{P1})}--\emph{(\ref{P3})}. Furthermore, the solution $u$ is also continuously dependent on the initial data in the $L^2$ norm.
\end{thm}

The rest of this paper is organized as follows. In Sect. \ref{preliminaries}, we establish an auxiliary problem related to (\ref{P1})--(\ref{P3}) and recall some preliminaries on Young measure. Sect. \ref{proof-of-P_w} is devoted to showing the existence of Young measure solutions to the auxiliary problem. In Sect. \ref{proof-of-main-thm}, we present the proof of Theorem \ref{thirdtheorem}. In Sect. \ref{Numerical experiments}, we show some numerical results
to demonstrate the ability of denoising by our model.
\section{Preliminaries}\label{preliminaries}
In this section, we recall some useful definitions and results related to Young measure.
We denote by $C_0(\mathbb{R}^d)$ the closure of the set of continuous functions on $\mathbb{R}^d$ with compact support. The dual of $C_0(\mathbb{R}^d)$ can be identified with the space $\mathcal{M}(\mathbb{R}^d)$ of signed Radon measures with finite mass via the pairing
$$
\left\langle\mu,f\right\rangle=\int_{\mathbb{R}^d}fd\mu,\quad f\in C_0(\mathbb{R}^d),~\mu\in\mathcal{M}(\mathbb{R}^d).
$$
Let $D\subseteq\mathbb{R}^d$ be a measurable set with finite measure, denoted by $\text{meas}(D)$. A map $\nu:D\rightarrow\mathcal{M}(\mathbb{R}^d)$ is called weakly$^\ast$ measurable if
$$
\left\langle\nu,f\right\rangle:D\rightarrow\mathbb{R}, x\mapsto\int_{\mathbb{R}^d}fd\nu_x
$$
is measurable for each $f\in C_0(\mathbb{R}^d)$, where $\nu_x=\nu(x)$. 


\begin{lem}\!\emph{(Fundamental theorem on Young measure)}\!\emph{\cite{ball1989}}\label{Fundamental-Theorem-on-YoungMeasure} Let $z_j:D\rightarrow\mathbb{R}^d$ \emph{($j=1,2,\cdots$)} be a sequence of measurable functions. Then there exists a
subsequence $\{z_{j_k}\}_{k=1}^{\infty}$ and a weakly$^\ast$ measurable map $\nu:D\rightarrow\mathcal{M}(\mathbb{R}^d)$ such that
    \begin{itemize}
        \item [(i)] $\nu(x)\geq 0$, $\left\|\nu(x)\right\|_{\mathcal{M}(\mathbb{R}^d)}=\int_{\mathbb{R}^d}d\nu_x\leq 1$ a.e. $x\in D$\emph{;}

        \item [(ii)] $\varphi(z_{j_k})$ converges weakly $\ast$ to $\left\langle\nu,\varphi\right\rangle$ in $L^{\infty}(D)$ for any $\varphi\in C_0(\mathbb{R}^d)$\emph{;}

        \item [(iii)] Furthermore, one has $\left\|\nu(x)\right\|_{\mathcal{M}(\mathbb{R}^d)}=1$ a.e. $x\in D$ if and only if
	$$
	\lim_{L\rightarrow\infty}\sup_{k\in\mathbb{N}}\emph{meas} \left\{x\in D:\left|z_{j_k}(x)\right|\geq L\right\}=0.
	$$
    \end{itemize}
\end{lem}
\begin{defn}
	The map $\nu:D\rightarrow\mathcal{M}(\mathbb{R}^d)$ in Lemma \ref{Fundamental-Theorem-on-YoungMeasure} is called the Young measure in $\mathbb{R}^d$ generated by the sequence $\{z_{j_k}\}_{k=1}^{\infty}$.
\end{defn}
For $p \geq 1$, we define the spaces as follows:
$$
\mathcal{E}_0^p(\mathbb{R}^d)=\left\{\varphi \in C(\mathbb{R}^d): \lim _{\left|\xi\right|\rightarrow+\infty} \frac{|\varphi(\xi)|}{1+|\xi|^p} \text { exists }\right\},
$$
$$
\mathcal{E}^p(\mathbb{R}^d)=\left\{\varphi \in C(\mathbb{R}^d): \sup _{\xi\in \mathbb{R}^d} \frac{|\varphi(\xi)|}{1+|\xi|^p}<+\infty\right\}.
$$
As noted in \cite{kinderlehrer1994}, $\mathcal{E}_0^p(\mathbb{R}^d)$ is a separable Banach space, while $\mathcal{E}^p(\mathbb{R}^d)$ is an inseparable space under the norm
$$
\left\|\varphi\right\|_{\mathcal{E}^p(\mathbb{R}^d)}=\sup _{\xi\in \mathbb{R}^d} \frac{|\varphi(\xi)|}{1+|\xi|^p},\quad \varphi\in\mathcal{E}^p(\mathbb{R}^d).
$$
\begin{defn}\!\!\!\cite{demoulini1996}\label{W1p-gradient-YoungMeasure}
A Young measure $\nu=\left(\nu_x\right)_{x \in D}$ on $\mathbb{R}^d$ is called a $W^{1,p}(D)$-gradient Young measure for some $p\geq 1$ if
    \begin{itemize}
    \item [(i)] 
    for each bounded continuous function $f$ in $\mathbb{R}^d$, $\left\langle\nu, f\right\rangle$ is measurable in $D$;
    \item [(ii)]
    there is a sequence of functions $\{u^k\}_{k=1}^{\infty}\subset W^{1, p}(D)$ for which the representation formula
	\begin{equation}
		\lim _{k \rightarrow \infty} \int_E \psi(\nabla u^k(x))dx=\int_E\left\langle \nu_x, \psi\right\rangle dx\label{representation-formula-YoungMeasure}
	\end{equation}
	holds for all measurable $E \subseteq D$ and all $\psi \in \mathcal{E}_0^p(\mathbb{R}^d)$. 
    \end{itemize}
	
	We also say $\nu$ the $W^{1, p}(D)$-gradient Young measure generated by $\{\nabla u^k\}_{k=1}^{\infty}$ and $\{\nabla u^k\}_{k=1}^{\infty}$ the $W^{1, p}(D)$-gradient generating sequence of $\nu$. In addition, the representation formula (\ref{representation-formula-YoungMeasure}) also holds for $\psi\in\mathcal{E}^p(\mathbb{R}^d)$. 
\end{defn}
\begin{defn}\!\!\!\cite{ballmurat1989}\label{convergence-in-the-sense-of-biting}
	Let $\{z^k\}_{k=1}^{\infty}\subset L^1(D)$ and $z\in L^1(D)$. We say that $\{z^k\}_{k=1}^{\infty}$ converges to $z$ in the biting sense if there is a decreasing sequence of
	subsets $E_{j+1}\subseteq E_j$ of $D$ with $\lim_{j\rightarrow\infty}\emph{meas}(E_j)=0$ such that $\{z^k\}_{k=1}^{\infty}$ converges weakly to $z$ in $L^1(D\backslash E_j)$ for all $j$.
\end{defn}
Kinderlehrer and Pedregal showed a property that characterizes $W^{1,p}$-gradient Young measures, as detailed in the following lemma.
\begin{lem}\!\!\!\emph{\cite{kinderlehrer1994}}
	Let $\nu=(\nu_x)_{x\in D}$ be a Young measure on $\mathbb{R}^d$. Then $\nu=(\nu_x)_{x\in D}$ is a $W^{1,p}(D)$-gradient Young measure if and only if
    \begin{itemize}
        \item [(i)] there exists $u\in W^{1,p}(D)$ such that 
	$$
	\nabla u(x)=\left\langle\nu_x,\emph{id}\right\rangle,\quad \emph{a.e.}~x\in D,
	$$
	where $\emph{id}$ is the unit mapping in $\mathbb{R}^d$\emph{;}

        \item [(ii)] Jensen's inequality
	$$
	\phi(\nabla u(x))\leq\left\langle\nu_x,\phi\right\rangle,\quad \emph{a.e.}~x\in D,
	$$
	holds for all $\phi\in\mathcal{E}^p(\mathbb{R}^d)$ continuous, quasiconvex, and bounded below\emph{;}

        \item [(iii)] $\left\langle\nu_x,\phi_p\right\rangle\in L^1(D)$, where $\phi_p(\xi)=\left|\xi\right|^p$, $\xi\in\mathbb{R}^d$.
    \end{itemize}
\end{lem}
\begin{lem}\!\!\!\emph{\cite{kinderlehrer1992}}\label{lower-semicontinuous-for-quasiconvex}
	Suppose $g\in\mathcal{E}^p(\mathbb{R}^d)$, for $p\geq 1$, is quasiconvex and bounded below and let $\{u^k\}_{k=1}^{\infty}$ converge weakly to $u$ in $W^{1, p}(D)$. Then,
	\begin{itemize}
	    \item [(i)] for all measurable $E\subseteq D$,
	$$
	\int_E g(\nabla u)dx\leq\liminf_{k\rightarrow\infty}\int_E g(\nabla u^k)dx;
	$$

            \item [(ii)] if, in addition,
	$$
	\lim_{k\rightarrow\infty} \int_D g(\nabla u^k)dx=\int_D g(\nabla u)dx,
	$$
	then $\{g(\nabla u^k)\}_{k=1}^{\infty}$ are weakly sequentially precompact in $L^1(D)$ and the sequence converges weakly to $g(\nabla u)$.
	\end{itemize}
\end{lem}
\begin{lem}\!\!\!\emph{\cite{kinderlehrer1992}}\label{criterion-of-W1p-gradient-YoungMeasure}
	Let $g$ and $\{u^k\}_{k=1}^{\infty}$ be as in Lemma \ref{lower-semicontinuous-for-quasiconvex} \emph{(ii)} and assume in addition that
	$$
	\left(c\left|\xi\right|^p-1\right)^{+} \leq g(\xi) \leq C\left|\xi\right|^p+1
	$$
	for $0<c\leq C$. Let $\nu=(\nu_x)_{x\in D}$ be generated by the gradients $\{\nabla u^k\}_{k=1}^{\infty}$. Then $\nu$ is a $W^{1, p}\left(D\right)$-gradient Young measure.
\end{lem}
\begin{lem}\!\!\!\emph{\cite{yin2003}}\label{WangChunpeng-YoungMeasure}
	Let $1<p\leq 2$. Suppose that $\{\nu^\alpha=(\nu_x^\alpha)_{x\in D}\}_{\alpha>0}$ is a family of $W^{1, p}(D)$-gradient Young measures, and each is generated by $\{\nabla u^{\alpha, m}\}_{m=1}^{\infty}$, where $u^{\alpha, m}$ is in $W^{1, p}(D)$ uniformly bounded in $\alpha$ and $m$. Then there exists a subsequence of $\{\nu^\alpha\}_{\alpha>0}$, denoted by $\{\nu^{\alpha_i}\}_{i=1}^{\infty}$, and a $W^{1, p}(D)$-gradient Young measure $\nu=(\nu_x)_{x\in D}$ such that
        \begin{itemize}
            \item [(i)] $\{\nu^{\alpha_i}\}_{i=1}^{\infty}$ converges weakly$^\ast$ to $\nu$ in $L^{\infty}(D;\mathcal{M}(\mathbb{R}^d))$, namely, $\{\left\langle\nu^{\alpha_i}, \psi\right\rangle\}_{i=1}^{\infty}$ converges weakly$^\ast$ to $\langle\nu,\psi\rangle$ in $L^{\infty}(D)$ for all $\psi\in C_0(\mathbb{R}^d)$\emph{;}

            \item [(ii)] $\{\nu^{\alpha_i}\}_{i=1}^{\infty}$ converges weakly to $\nu$ in $L^1(D ;(\mathcal{E}_0^q(\mathbb{R}^d))^{\prime})$ for $1\leq q<p$, namely, $\{\left\langle v^{\alpha_i}, \psi\right\rangle\}_{i=1}^{\infty}$ converges weakly to $\langle\nu,\psi\rangle$ in $L^1(D)$ for all $\psi\in \mathcal{E}_0^q(\mathbb{R}^d)$\emph{;}

            \item [(iii)] $\{\nu^{\alpha_i}\}_{i=1}^{\infty}$ converges to $\nu$ in $L^1(D ;(\mathcal{E}_0^p(\mathbb{R}^d))^{\prime})$ in the biting sense, namely, $\{\left\langle\nu^{\alpha_i}, \psi\right\rangle\}_{i=1}^{\infty}$ converges to $\langle \nu, \psi\rangle$ in the biting sense for all $\psi\in\mathcal{E}_0^p(\mathbb{R}^d)$.
        \end{itemize}
\end{lem}
\begin{rem}\label{firstremark}
	Suppose $(\nu^{\alpha})_{\alpha>0}$, with $\nu^{\alpha}=(\nu^{\alpha}_{x})_{x\in D}$, is a sequence of Young measures bounded in $L^1(D;(\mathcal{E}_0^p(\mathbb{R}^d))^{\prime})$. For each $\alpha$, let $\{\nabla v^{\alpha,k}\}_{k=1}^{\infty}$ be the generating gradients, where $v^{\alpha,k}\in W^{1,p}(D)$. Then $(\nu^{\alpha})_{\alpha>0}$ is bounded in $L^{\infty}(D;\mathcal{M}(\mathbb{R}^d))$ and a subsequence of the $\{v^{\alpha,k}\}_{\alpha,k}$ is bounded in $W^{1,p}(D)$ uniformly in $\alpha$ and $k$ (and can be taken to be the new generating sequence). Thus, Lemma \ref{WangChunpeng-YoungMeasure} applies.
\end{rem}

We now define the Young measure solution to the problem (\ref{P1})--(\ref{P3}). Without loss of generality, we assume $\lambda=1$ in (\ref{P1}). Denote $\varphi^{**}$ the convexification of $\varphi$, namely,
$$
\varphi^{**}(\xi)=\sup\left\{g(\xi):g\leq\varphi,g\text{ convex}\right\}.
$$
Since $\varphi\in C^1(\mathbb{R}^n)$, it follows that $\varphi^{**}\in C^1(\mathbb{R}^n)$ is convex. Set 
$$
q^{**}=\nabla \varphi^{**}.
$$ 
Observe that
$q=q^{**}$ on the set $\left\{\xi\in\mathbb{R}^n:\varphi(\xi)=\varphi^{**}(\xi)\right\}$. Furthermore, $\varphi^{**}$ and $q^{**}$ satisfy the same structure conditions (\ref{structure-condition}) as $\varphi$ and $q$, respectively, i.e.
\begin{equation}
	\gamma_1\left|\xi\right|^p\leq \varphi^{**}(\xi)\leq \gamma_2\left|\xi\right|^p+1,\quad \left|q^{**}(\xi)\right|\leq \gamma_2\left|\xi\right|^{p-1}+1,\quad\xi\in \mathbb{R}^n.\label{growth-condition-for-varphi*-and-p}
\end{equation}
Given fixed but arbitrary constants $a<b$, we denote $W^{1,p}_{a,b}(\Omega)$ as the subspace of $W^{1,p}(\Omega)$ functions that satisfy $a\leq u(x)\leq b$ almost everywhere in $\Omega$, namely,
$$
W^{1,p}_{a,b}(\Omega)=\left\{u\in W^{1,p}(\Omega): a\leq u(x)\leq b,~\mathrm{a.e.}~x\in\Omega\right\}.
$$
\begin{defn}
	A function $u\in L^{\infty}(0,T;W^{1,p}(\Omega))\cap L^{\infty}(Q_T)$ with $\frac{\partial u}{\partial t}\in L^2(Q_T)$ is said to be a
	Young measure solution to the problem (\ref{P1})--(\ref{P3}), if there exists a $W^{1,p}(Q_T)$-gradient Young measure $\nu=(\nu_{x,t})_{(x,t)\in Q_T}$ on $\mathbb{R}^n$ such that
	\begin{equation}
		\iint_{Q_T}\left(\alpha(x)\langle\nu, q\rangle\cdot\nabla\eta+\frac{\partial u}{\partial t}\eta+\frac{u-f}{u^2}\eta\right) dxdt=0,\label{equilibrium-equation-for-u}
	\end{equation}
	for any $\eta\in C^{\infty}(\overline{Q_T})$, and
	\begin{equation}
		\nabla u(x,t)=\langle\nu_{x,t},\mathrm{id}\rangle,\quad\text{a.e.}~(x,t)\in Q_T,\label{gradient-of-u}
	\end{equation}
	\begin{equation}
		\langle\nu_{x,t},q\cdot \mathrm{id}\rangle=\langle\nu_{x,t}, q\rangle\cdot\langle\nu_{x,t}, \mathrm{id}\rangle,\quad\text{a.e.}~(x,t)\in Q_T,\label{independence-for-u}
	\end{equation}
	\begin{equation}
		\operatorname{supp}\nu_{x,t}\subseteq\left\{\xi\in \mathbb{R}^n: \varphi(\xi)=\varphi^{**}(\xi)\right\},\quad\text{a.e.}~(x,t)\in Q_T,\label{support-of-u}
	\end{equation}
	and (\ref{P3}) holds in the sense of trace.
\end{defn}

Our purpose is to construct a Young measure solution to the problem (\ref{P1})--(\ref{P3}) by utilizing Schauder’s fixed-point theorem \cite{evans2022partial}.
Considering this, we first introduce the solution space $W(0,T)$ for the problem (\ref{P1})--(\ref{P3}), defined as follows:
$$
W(0,T)=\left\lbrace w\in L^{\infty}\left(0,T;W^{1,p}(\Omega)\right), \frac{\partial w}{\partial t}\in L^2(Q_T)\right\rbrace,
$$
which is a Banach space equipped with the graph norm \cite{brezis2010}. In the following, we assume that $\alpha(x)\geq\alpha_1>0$ for $x\in\Omega$, and the initial value $f\in W^{1,p}(\Omega)$ satisfies
\begin{equation}
    0<l:=\inf_{x\in\Omega}f(x)~\text{and}~d:=\sup_{x\in\Omega}f(x)<\infty.\label{assumption-for-u0}
\end{equation}

Take $w\in W(0,T)\cap L^{\infty}(Q_T)$ such that
\begin{equation}
	l\leq w(x,t)\leq d,\quad\text{a.e.}~(x,t)\in Q_T.\label{l<w<d}
\end{equation}
Let us denote the following auxiliary problem by ($\mathcal{P}_w$)
\begin{numcases}{}
	\frac{\partial u}{\partial t}=\operatorname{div} \left(\alpha(x)q(\nabla u) \right)+\frac{f-u}{w^2}, & $(x,t)\in Q_T$,\notag\\
	\langle \nabla u,\vec{n} \rangle=0, & $(x,t)\in S_T$,\notag \\
	u(x,0)=f(x) & $x\in\Omega$.\notag
\end{numcases}


Next, we define the Young measure solution to problem ($\mathcal{P}_w$).
\begin{defn}
	We say $u_w\in W(0,T)\cap L^{\infty}(Q_T)$ is the Young measure solution of problem ($\mathcal{P}_w$), if there exists a $W^{1,p}(Q_T)$-gradient Young measure $\nu^w=(\nu^w_{x,t})_{(x,t)\in Q_T}$ on $\mathbb{R}^n$ such that
	\begin{equation}
		\iint_{Q_T}\left(\alpha(x)\langle\nu^w, q\rangle\cdot\nabla\eta+\frac{\partial u_w}{\partial t}\eta+\frac{u_w-f}{w^2}\eta\right) dxdt=0,\label{equilibrium-equation-for-uw}
	\end{equation}
	for any $\eta\in C^{\infty}(\overline{Q_T})$, and
	\begin{equation}
		\nabla u_w(x,t)=\langle\nu^w_{x,t},\mathrm{id}\rangle,\quad\text{a.e.}~(x,t)\in Q_T,\label{gradient-of-uw}
	\end{equation}
	\begin{equation}
		\langle\nu^w_{x,t},q\cdot \mathrm{id}\rangle=\langle\nu^w_{x,t}, q\rangle\cdot\langle\nu^w_{x,t}, \mathrm{id}\rangle,\quad\text{a.e.}~(x,t)\in Q_T,\label{independence-for-uw}
	\end{equation}
	\begin{equation}
		\operatorname{supp}\nu^w_{x,t}\subseteq\left\{\xi\in \mathbb{R}^n: \varphi(\xi)=\varphi^{**}(\xi)\right\},\quad\text{a.e.}~(x,t)\in Q_T,\label{support-of-uw}
	\end{equation}
	and
	\begin{equation}
		u_w(x,0)=u_0(x),\quad x\in\Omega,\label{trace-of-uw}
	\end{equation}
	in the sense of trace.
\end{defn}

\section{Existence of Young measure solutions to $(\mathcal{P}_w)$}\label{proof-of-P_w}
In this section, we will study the existence of Young measure solutions to problem $(\mathcal{P}_w)$, inspired by the methods of Demoulini \cite{demoulini1996}, Yin and Wang \cite{yin2003}, and Guidotti \cite{guidotti2012backward}.

Let $m$ be a positive integer. For $v\in W^{1,p}(\Omega)$, we define
$$
E_m(v;u^{j-1}_m)=\int_{\Omega} \left(\alpha(x)\varphi(\nabla v)+\frac{m}{2T}(v-u_m^{j-1})^2+\frac{(v-f)^2}{2(w_m^j)^2}\right)dx,
$$
where $u_m^0=f$, $u_m^{j-1}\in W^{1,p}_{l,d}(\Omega)$, $w_m^j(x)=w(x,jT/m),~\text{a.e.}~x\in\Omega$, $1\leq j\leq m$. Since $\varphi$ is not necessarily convex, and following relaxation theory, $\varphi$ is replaced by its convexification $\varphi^{**}$. The energy functional $E_m$ has the same infimum as
$$
E_m^{**}(v;u^{j-1}_m)=\int_{\Omega} \left(\alpha(x)\varphi^{**}(\nabla v)+\frac{m}{2T}(v-u_m^{j-1})^2+\frac{(v-f)^2}{2(w_m^j)^2}\right)dx.
$$
The two functionals are both lower bounded owing to (\ref{structure-condition}) and (\ref{growth-condition-for-varphi*-and-p}). It is readily seen that
$$
E_m^{**}(T_l^d(v);f)\leq E_m^{**}(v;f),\quad v\in W^{1,p}(\Omega),
$$
and that
$$
E_m(T_l^d(v);f)\leq E_m(v;f),\quad v\in W^{1,p}(\Omega),
$$
where the truncation operation $T_l^d:\mathbb{R}\rightarrow\mathbb{R}^{+}$ is defined by
$$
T_l^d(z)=
\left\{
\begin{aligned}
	&d, \quad &z>d,\\
	&z, \quad &l\leq z\leq d,\\
	&l, \quad &z<l.\\
\end{aligned}
\right.
$$
It follows that minimizers and minimizing sequences can all be assumed to lie in $W^{1,p}_{l,d}(\Omega)$ without loss of generality. 
First, we prove the following lemma.
\begin{lem}\label{firstlemma}
	There exists $u_m^j\in W_{l,d}^{1,p}(\Omega)$ such that $u_m^j$ is a minimum of $E_m^{**}(v;u^{j-1}_m)$ and
	\begin{equation}
		\left\|u_m^j\right\|_{W^{1,p}(\Omega)}\leq M_1,\label{boundness-of-u_m^j-for-W1pnorm-and-L^inftynorm}
	\end{equation}
	moreover,
	\begin{equation}
		\frac{m}{T}\sum_{j=1}^{m}\left\|u_m^{j}-u_m^{j-1}\right\|_{L^2(\Omega)}^2\leq M_2,\label{boundness-for-m/4T-sum-(u_m^j-u_m^j-1)^2}
	\end{equation}
	where $M_i~(i=1,2)$ are positive constants independent of $m$ and $j$.
\end{lem}
\begin{proof}
	By relaxation theorem \cite{dacorogna2007}, we get that
	\begin{equation}
		\inf\left\{E_m(v;u_m^{j-1}):v\in W^{1,p}(\Omega)\right\}=\inf\left\{E_m^{**}(v;u_m^{j-1}):v\in W^{1,p}(\Omega)\right\}.\label{relaxation-infimum}
	\end{equation}
	Let $\{u_m^{j,k}\}_{k=1}^{\infty}\subset W^{1,p}_{l,d}(\Omega)$ be a minimizing sequence of $E_m(\cdot;u_m^{j-1})$ satisfying
	\begin{equation}
		E_m(u_m^{j,k};u_m^{j-1})<\inf\left\{E_m(v;u_m^{j-1}):v\in W^{1,p}(\Omega)\right\}+\frac{1}{mk},~k\geq 1.\label{relaxation-inequality-for-Em}
	\end{equation}
	Since $\varphi^{**}\leq\varphi$, (\ref{relaxation-infimum}) and (\ref{relaxation-inequality-for-Em}) yield
	\begin{equation}
		E_m^{**}(u_m^{j,k};u_m^{j-1})<\inf\left\{E_m^{**}(v;u_m^{j-1}):v\in W^{1,p}(\Omega)\right\}+\frac{1}{mk},~k\geq 1,\label{relaxation-inequality-for-Em*}
	\end{equation}
	which, together with the Young's inequality, implies that
	$$
	\begin{aligned}
		&\int_{\Omega} \left(\alpha(x)\varphi^{**}(\nabla u_m^{j,k})+\frac{m}{2T}(u_m^{j,k}-u_m^{j-1})^2\right)dx-\frac{1}{mk}\\
		\leq &\int_{\Omega} \left(\alpha(x)\varphi^{**}(\nabla u_m^{j-1})+\frac{(u_m^{j-1}-f)^2}{2(w_m^j)^2}-\frac{(u_m^{j,k}-f)^2}{2(w_m^j)^2}\right)dx\\
		= & \int_{\Omega} \left(\alpha(x)\varphi^{**}(\nabla u_m^{j-1})+\frac{u_m^{j-1}+u_m^{j,k}-2f}{2(w_m^j)^2}(u_m^{j-1}-u_m^{j,k})\right)dx\\
		\leq & \int_{\Omega} \alpha(x)\varphi^{**}(\nabla u_m^{j-1})dx+\frac{T}{m}\int_{\Omega}\frac{(u_m^{j-1}+u_m^{j,k}-2f)^2}{4(w_m^j)^{4}}dx\\
		&+\frac{m}{4T} \int_{\Omega}(u_m^{j,k}-u_m^{j-1})^2dx,\quad k\geq 1.
	\end{aligned}
	$$
	Using this formula, (\ref{growth-condition-for-varphi*-and-p}), (\ref{l<w<d}) and the Poincar\'{e} inquality, one gets that
	\begin{align}
		&\int_{\Omega} \left(\alpha(x)\varphi^{**}(\nabla u_m^{j,k})+\frac{m}{4T}(u_m^{j,k}-u_m^{j-1})^2\right)dx\notag\\
		\leq & \int_{\Omega} \alpha(x)\varphi^{**}(\nabla u_m^{j-1})dx+\frac{T}{m}\int_{\Omega}\frac{(u_m^{j-1}+u_m^{j,k}-2f)^2}{4(w_m^j)^{4}}dx+\frac{1}{mk}\notag\\
		\leq & \int_{\Omega} \alpha(x)\varphi^{**}(\nabla u_m^{j-1})dx+\frac{T}{m}\int_{\Omega} \frac{(u_m^{j-1})^2+(u_m^{j,k})^2+2f^2}{2l^{4}}dx+\frac{1}{mk}\notag\\
		\leq & \int_{\Omega} \alpha(x)\varphi^{**}(\nabla u_m^{j-1})dx+\frac{C_1}{m}\int_{\Omega} (u_m^{j-1})^2dx+\frac{C_1}{m}+\frac{1}{mk}\notag\\
		\leq & \int_{\Omega} \alpha(x)\varphi^{**}(\nabla u_m^{j-1})dx+\frac{C_2}{m}\int_{\Omega}d^{2-p}\left|u_m^{j-1}\right|^p dx+\frac{C_2}{m}\notag\\
		\leq & \int_{\Omega} \alpha(x)\varphi^{**}(\nabla u_m^{j-1}) dx+\frac{C_3}{m}\int_{\Omega}\left|\nabla u_m^{j-1}\right|^pdx+\frac{C_3}{m}\notag\\
		\leq & \left(1+\frac{C_0}{m}\right) \int_{\Omega} \alpha(x)\varphi^{**}(\nabla u_m^{j-1})dx+\frac{C_0}{m},\quad 1\leq j\leq m,\ k\geq 1,\label{iteration-inequality}
	\end{align}
	where $C_1$, $C_2$, $C_3$ and $C_0$ depend only on $\Omega$, $T$, $\alpha_1$, $\gamma_1$, $p$, $l$ and $d$. 
	
	Take $u_m^0=f$. Then (\ref{iteration-inequality}) with $j=1$ shows that $\{u_m^{1,k}\}_{k=1}^{\infty}$ is uniformly bounded in $W^{1,p}(\Omega)\cap L^{\infty}(\Omega)$. By the growth condition (\ref{growth-condition-for-varphi*-and-p}) and the Rellich theorem \cite{evans2022partial}, there exists a function $u_m^1\in W^{1,p}_{l,d}(\Omega)\cap L^{\infty}(\Omega)$ and
	a convergent subsequence of $\{u_m^{1,k}\}_{k=1}^{\infty}$, denoted by itself, such that as $k\rightarrow\infty$
	$$
	u_m^{1,k}\rightarrow u_m^1~\text{strongly in}~L^p(\Omega),\quad\nabla u_m^{1,k}\rightharpoonup\nabla u_m^1~\text{weakly in}~L^p(\Omega).
	$$
	By virtue of Theorem 2.4 in \cite{acerbi1984}, we get that $E_m^{**}(\cdot;f)$ is sequentially weakly lower semi-continuous in $W^{1,p}(\Omega)$. Hence $u_m^1$ is a minimum of $E_m^{**}(\cdot;f)$ and
	\begin{align}
		\int_{\Omega}\alpha(x)\varphi^{**}(\nabla u_m^1)dx&=\lim_{k\rightarrow\infty}\int_{\Omega}\alpha(x)\varphi^{**}(\nabla u_m^{1,k})dx\notag\\
		&=\lim_{k\rightarrow\infty}\int_{\Omega}\alpha(x)\varphi(\nabla u_m^{1,k})dx,\notag
	\end{align}
	\begin{align}
		&\int_{\Omega}\left(\alpha(x)\varphi^{**}(\nabla u_m^1)+\frac{m}{4T}(u_m^1-f)^2\right)dx\notag\\
		\leq & \left(1+\frac{C_0}{m}\right) \int_{\Omega} \alpha(x)\varphi^{**}(\nabla f)dx+\frac{C_0}{m}.\notag
	\end{align}
	Repeating the above process in turn, we obtain $u_m^j\in W^{1,p}_{l,d}(\Omega)\cap L^{\infty}(\Omega)$ for $j=1,2,\cdots,m$, which satisfies
	\begin{equation}
		u_m^{j,k}\rightarrow u_m^j~\text{strongly in}~L^p(\Omega),\quad\nabla u_m^{j,k}\rightharpoonup\nabla u_m^j~\text{weakly in}~L^p(\Omega),\label{convergence-of-u_m^jk}
	\end{equation}
	as $k\rightarrow\infty$, and
	\begin{equation}
		E_m^{**}(u_m^j;u_m^{j-1})=\inf\left\{E_m^{**}(v;u_m^{j-1}):v\in W^{1,p}(\Omega)\right\}.\label{u_m^j-minimum-E_m^*}
	\end{equation}
	Similarly, we have
	\begin{align}
		\int_{\Omega}\alpha(x)\varphi^{**}(\nabla u_m^j)dx&=\lim_{k\rightarrow\infty}\int_{\Omega}\alpha(x)\varphi^{**}(\nabla u_m^{j,k})dx\notag\\
		&=\lim_{k\rightarrow\infty}\int_{\Omega}\alpha(x)\varphi(\nabla u_m^{j,k})dx,\label{convergence-for-the-first-term-of-E_m*-and-E_m}
	\end{align}
	\begin{align}
		&\int_{\Omega} \left(\alpha(x)\varphi^{**}(\nabla u_m^{j})+\frac{m}{4T}(u_m^{j}-u_m^{j-1})^2\right)dx\notag\\
		\leq& \left(1+\frac{C_0}{m}\right)\int_{\Omega} \alpha(x)\varphi^{**}(\nabla u_m^{j-1})dx+\frac{C_0}{m}.\label{iteration-inequlity-of-E_m*}
	\end{align}
	Direct calculation shows that
	$$
	\left(1+\frac{C_0}{m}\right)^m\leq e^{C_0},\quad \sum_{i=0}^{m-1} \left(1+\frac{C_0}{m}\right)^i\leq\frac{m}{C_0}e^{C_0}.
	$$
	Then, one gets from (\ref{iteration-inequality}) and (\ref{iteration-inequlity-of-E_m*}) that
	\begin{equation}
		\int_{\Omega} \alpha(x)\varphi^{**}(\nabla u_m^{j,k})dx\leq e^{C_0}\int_{\Omega}\alpha(x)\varphi^{**}(\nabla f)dx+e^{C_0},~k\geq 1,\label{iteration-inequlity-for-the-first-term-of-E_m*-with-u_m^jk}
	\end{equation}
	\begin{equation}
		\int_{\Omega} \alpha(x)\varphi^{**}(\nabla u_m^j)dx\leq e^{C_0}\int_{\Omega}\alpha(x)\varphi^{**}(\nabla f)dx+e^{C_0},~1\leq j\leq m.\label{iteration-inequlity-for-the-first-term-of-E_m*-with-u_m^j}
	\end{equation}
	Summing up (\ref{iteration-inequlity-of-E_m*}) from $1$ to $m$ leads to
	\begin{align}
		&\sum_{j=1}^{m}\int_{\Omega} \alpha(x)\varphi^{**}(\nabla u_m^{j})dx+\frac{m}{4T}\sum_{j=1}^{m}\int_{\Omega}(u_m^{j}-u_m^{j-1})^2dx\notag\\
		\leq&\left(1+\frac{C_0}{m}\right) \sum_{j=1}^{m}\int_{\Omega} \alpha(x)\varphi^{**}(\nabla u_m^{j-1})dx+C_0.\notag
	\end{align}
	Combining the above inequality with (\ref{iteration-inequlity-for-the-first-term-of-E_m*-with-u_m^j}), we get
	\begin{align}
		&\frac{m}{4T}\sum_{j=1}^{m}\int_{\Omega}(u_m^{j}-u_m^{j-1})^2dx\notag\\
		\leq& \frac{C_0}{m}\sum_{j=1}^{m}\int_{\Omega} \alpha(x)\varphi^{**}(\nabla u_m^{j-1})dx+\int_{\Omega} \alpha(x)\varphi^{**}(\nabla f)dx+C_0\notag\\
		\leq& (C_0e^{C_0}+1)\int_{\Omega}\alpha(x)\varphi^{**}(\nabla f)dx+C_0e^{C_0}+C_0.\label{boundness-for-sum-of-[u_m^j-u_m^j-1]}
	\end{align}
	This completes the proof.
\end{proof}
Recalling Lemma 1.2 in \cite{demoulini1996high} and (\ref{u_m^j-minimum-E_m^*}), together with the $W^{1,p}$-weak sequential lower semi-continuity of $E_m^{**}(\cdot;u_m^{j-1})$, it can be deduced that
$$
\varphi^{**}(\nabla u_m^{j,k})\rightharpoonup\varphi^{**}(\nabla u_m^j)~\text{weakly in}~L^1(\Omega),~\text{as}~k\rightarrow\infty.
$$
For $j=1,2,\cdots,m$, denote $\nu^{m,j}=(\nu^{m,j}_x)_{x\in\Omega}$ the Young measure generated by $\{\nabla u_m^{j,k}\}_{k=1}^{\infty}$. By Lemma \ref{criterion-of-W1p-gradient-YoungMeasure} with (\ref{growth-condition-for-varphi*-and-p}), $\nu^{m,j}$ is a $W^{1,p}(\Omega)$-gradient Young measure and the sequence $\{\varphi(\nabla u_m^{j,k})\}_{k=1}^{\infty}$ is also weakly convergent in $L^1(\Omega)$; the representation formula (\ref{representation-formula-YoungMeasure}) describes the weak $L^1$ limits and by (\ref{convergence-for-the-first-term-of-E_m*-and-E_m}) we obtain
$$
\int_{\Omega}\alpha(x)\varphi^{**}(\nabla u_m^j)dx=\int_{\Omega}\alpha(x)\langle\nu^{m,j}, \varphi^{**}\rangle dx=\int_{\Omega}\alpha(x)\langle\nu^{m,j}, \varphi\rangle dx,
$$
which together with $\varphi^{**}\leq\varphi$, implies
$$
\varphi^{**}(\nabla u_m^j(x))=\langle\nu_x^{m,j},\varphi^{**}\rangle=\langle\nu_x^{m,j},\varphi\rangle,\quad\text{a.e.}~x\in\Omega,
$$
and therefore
\begin{equation}
	\operatorname{supp}\nu_x^{m,j}\subseteq\left\{\xi\in\mathbb{R}^n:\varphi^{**}(\xi)=\varphi(\xi)\right\},\quad\text{a.e.}~x\in\Omega.\label{supportset-of-nu_m^j}
\end{equation}
Noticing that $\left\{\xi\in\mathbb{R}^n:\varphi(\xi)=\varphi^{**}(\xi)\right\}\subseteq\left\{\xi\in\mathbb{R}^n:q(\xi)=q^{**}(\xi)\right\}$, we see that
\begin{equation}
	\langle\nu_x^{m,j},q\rangle=\langle\nu_x^{m,j},q^{**}\rangle,\quad\text{a.e.}~x\in\Omega.\label{<nu_m^j,p>=<nu_m^j,q>}
\end{equation}
In addition, it follows from $\mathrm{id}\in\mathcal{E}_0^1(\mathbb{R}^n)$ and (\ref{convergence-of-u_m^jk}) that
\begin{equation}
	\nabla u_m^j(x)=\langle \nu_x^{m,j},\mathrm{id}\rangle,\quad\text{a.e.}~x\in\Omega.\label{gradient-of-u_m^j}
\end{equation}

Let $\chi_m^j$ the characteristic function of $[(j-1)T/m,jT/m)$ and denote
$$
\lambda_m^j(t)=\left(\frac{m}{T}t-(j-1)\right)\chi_m^j(t),\quad 0\leq t\leq T.
$$
Define
$$
\begin{aligned}
	&u_m(x,t)=\sum_{j=1}^{m}\chi_m^j(t)\left\{u_m^{j-1}(x)+\lambda_m^j(t)(u_m^j(x)-u_m^{j-1}(x))\right\}, & (x,t)\in Q_T,\\
	&v_m(x,t)=\sum_{j=1}^{m}\chi_m^j(t)u_m^j(x), & (x,t)\in Q_T,\\
	&w_m(x,t)=\sum_{j=1}^{m}\chi_m^j(t)w_m^j(x), & (x,t)\in Q_T,\\
	&v_m^k(x,t)=\sum_{j=1}^{m}\chi_m^j(t)u_m^{j,k}(x),~k=1,2,\cdots, &(x,t)\in Q_T,
\end{aligned}
$$
and
$$
\nu^m=(\nu_{x,t}^m)_{(x,t)\in Q_T},~\nu_{x,t}^m=\sum_{j=1}^{m}\chi_m^j(t)\nu_x^{m,j},\quad\quad\quad\quad\quad\quad\quad\quad\  (x,t)\in Q_T.
$$
From the definitions of $u_m^j$, $u_m^{j,k}$ and $\nu^{m,j}$, we get that $u_m$, $v_m$, $v_m^k\in  L^{\infty}(Q_T)\cap L^{\infty}(0,T;W_{l,d}^{1,p}(\Omega))$ with $u_m(x,0)=u_0(x),~\text{a.e.}~x\in\Omega$ and the $W^{1,p}(Q_T)$-gradient Young measure $\nu^{m}\in L^{1}(Q_T;(\mathcal{E}_0^p(\mathbb{R}^n))^{\prime})$ is generated by $\{\nabla v_m^k\}_{k=1}^{\infty}$. Moreover, (\ref{supportset-of-nu_m^j}) and (\ref{gradient-of-u_m^j}) yield
\begin{equation}
	\nabla v_m(x,t)=\langle \nu_{x,t}^m,\mathrm{id}\rangle,\quad\text{a.e.}~(x,t)\in Q_T,\label{gradientforum}
\end{equation}
\begin{equation}
	\operatorname{supp}\nu_{x,t}^m\subseteq\left\{\xi\in\mathbb{R}^n:\varphi^{**}(\xi)=\varphi(\xi)\right\},\quad\text{a.e.}~(x,t)\in Q_T.\label{suppform}
\end{equation}
To establish the existence of Young measure solutions to the problem ($\mathcal{P}_{w}$), we require uniform estimates for the approximate solutions $u_m$, $v_m$ and the Young measures $\nu^m$.
\begin{lem}\label{secondlemma}
	The approximate solution $u_m$, $v_m$, and the Young measure $\nu^m$ defined above satisfy
	\begin{equation}
		\iint_{Q_T}\left(\alpha(x)\langle\nu^m,q\rangle\cdot\nabla\eta+\frac{\partial u_m}{\partial t}\eta+\frac{v_m-f}{(w_m)^2}\eta\right)dxdt=0,\label{equilibrium-equation-of-u_m-and-nu^m}
	\end{equation}
	for any $\eta\in C^{\infty}(\overline{Q_T})$. Furthermore, there exists a positive constant $M$ that depends only on $\Omega$, $T$, $M_1$, $M_2$ and $\gamma_2$, but not on $m$, such that
	\begin{equation}
		\sup_{0\leq t\leq T}\left\|v_m\right\|_{W^{1,p}(\Omega)}\leq M,\quad l\leq v_m(x,t)\leq d,\label{boundness-for-v_m}
	\end{equation} 
	\begin{equation}
		\sup_{0\leq t\leq T}\left\|u_m\right\|_{W^{1,p}(\Omega)}\leq M,\quad l\leq u_m(x,t)\leq d,\label{boundness-for-u_m}
	\end{equation}
	\begin{equation}
		\left\|\frac{\partial u_m}{\partial t}\right\|_{L^2(Q_T)}\leq M,\quad \left\|u_m-v_m\right\|_{L^2(Q_T)}\leq \frac{M}{m},\label{boundness-for-du_m/dt-and-(u_m-v_m)}
	\end{equation}
	\begin{equation}
		\left\|\nu^m\right\|_{L^1(Q_T;(\mathcal{E}_0^p(\mathbb{R}^n))^{\prime})}\leq M,\quad\left\|\langle\nu^m,q\rangle\right\|_{L^{p/(p-1)}(Q_T)}\leq M.\label{boundness-for-|nu^m|_L^1-and-<nu^m,q>}
	\end{equation}
\end{lem}
\begin{proof}
	Let $\eta\in C^{\infty}(\overline{\Omega})$, $-1<\epsilon<1$. Then there
	exists $C>0$ such that
	$$
	\varphi^{**}(\xi+\epsilon\nabla\eta)\leq C(1+\left|\xi\right|^p),\quad\xi\in\mathbb{R}^n.
	$$
	Again by applying Theorem 2.4 in \cite{acerbi1984} and the representation formula (\ref{representation-formula-YoungMeasure}), we get
	\begin{align}
		&E_m^{**}(u_m^j;u^{j-1}_m)\notag\\
		=&\int_{\Omega} \left(\alpha(x)\left\langle\nu^{m,j},\varphi^{**} \right\rangle+\frac{m}{2T}(u_m^j-u_m^{j-1})^2+\frac{(u_m^j-f)^2}{2(w_m^j)^2}\right)dx\notag\\
		\leq&\int_{\Omega} \alpha(x)\varphi^{**}(\nabla u_m^j+\epsilon\nabla \eta)+\int_{\Omega}\frac{m}{2T}(u_m^j+\epsilon\eta-u_m^{j-1})^2 dx\notag\\
		&+\int_{\Omega}\frac{(u_m^j+\epsilon\eta-f)^2}{2(w_m^j)^2}dx\notag\\
		\leq&\liminf\limits_{k\rightarrow\infty}\int_{\Omega} \alpha(x)\varphi^{**}(\nabla u_m^{j,k}+\epsilon\nabla \eta)dx+\int_{\Omega}\frac{m}{2T}(u_m^j+\epsilon\eta-u_m^{j-1})^2 dx\notag\\
		&+\int_{\Omega}\frac{(u_m^j+\epsilon\eta-f)^2}{2(w_m^j)^2}dx\notag\\
		=&\int_{\Omega} \alpha(x)\left\langle\nu^{m,j},\varphi^{**}(\cdot+\epsilon\nabla\eta) \right\rangle dx+\int_{\Omega}\frac{m}{2T}(u_m^j+\epsilon\eta-u_m^{j-1})^2 dx\notag\\
		&+\int_{\Omega}\frac{(u_m^j+\epsilon\eta-f)^2}{2(w_m^j)^2}dx,\notag
	\end{align}
	which implies the equilibrium equation
	\begin{equation}
		\int_{\Omega}\left(\alpha(x)\langle\nu^{m,j},q^{**}\rangle\cdot\nabla\eta+\frac{m}{T}(u_m^j-u_m^{j-1})\eta+\frac{u_m^j-f}{(w_m^j)^2}\eta\right)dx=0,\label{equilibrium-equation-of-u_m^j-and-nu^m,j}
	\end{equation}
	for any $\eta\in C^{\infty}(\overline{\Omega})$. Setting the G\^{a}teaux derivative of $E_m^{**}(\cdot;u_m^{j-1})$ to zero at the minimum $u_m^j$, we obtain
	\begin{equation}
		\int_{\Omega}\left(\alpha(x)q^{**}(\nabla u_m^j)\cdot\nabla\eta+\frac{m}{T}(u_m^j-u_m^{j-1})\eta+\frac{u_m^j-f}{(w_m^j)^2}\eta\right)dx=0,\label{gateaux-derivative-of-E_m^*-at-u_m^j}
	\end{equation}
	for any $\eta\in C^{\infty}(\overline{\Omega})$. Then, (\ref{<nu_m^j,p>=<nu_m^j,q>}), (\ref{equilibrium-equation-of-u_m^j-and-nu^m,j}) and (\ref{gateaux-derivative-of-E_m^*-at-u_m^j}) show that
	\begin{equation}
		\langle\nu^{m,j}_x,q \rangle=\langle\nu^{m,j}_x,q^{**} \rangle=q^{**}(\nabla u_m^j(x)),\quad\text{a.e.}~x\in\Omega,\label{p(gradient-of-u_m^j)=<nu_m^j,p>=<nu_m^j,q>}
	\end{equation}
	\begin{equation}
		\int_{\Omega}\left(\alpha(x)\langle\nu^{m,j},q\rangle\cdot\nabla\eta+\frac{m}{T}(u_m^j-u_m^{j-1})\eta+\frac{u_m^j-f}{(w_m^j)^2}\eta\right)dx=0,\label{gateaux-derivative-of-E_m-at-u_m^j}
	\end{equation}
	for any $\eta\in C^{\infty}(\overline{\Omega})$. By Lemma \ref{firstlemma}, we get the estimate
	$$
	\begin{aligned}
		\left\|\langle\nu^{m,j},q\rangle\right\|_{L^{p/(p-1)}(\Omega)}&=\left\|q^{**}(\nabla u_m^j)\right\|_{L^{p/(p-1)}(\Omega)}\\
		&\leq2\max\{\mathrm{meas}(\Omega),1\}+ 2\gamma_2(M_1+1).\label{boundness-for-<nu_m^j,q>}
	\end{aligned}
	$$
	According to (\ref{gateaux-derivative-of-E_m-at-u_m^j}), we deduce
	$$
	\iint_{Q_T}\left(\alpha(x)\langle\nu^m,q\rangle\cdot\nabla\eta+\frac{\partial u_m}{\partial t}\eta+\frac{v_m-f}{w_m^2}\eta\right)dxdt=0,\quad \eta\in C^{\infty}(\overline{Q_T}).
	$$
	
	From the direct calculation, we see that
	$$
	\begin{aligned}
		\left\|\langle\nu^m,q\rangle\right\|_{L^{p/(p-1)}(Q_T)}\leq&\sup_{1\leq j\leq m}\left\|\langle\nu^{m,j},q\rangle\right\|_{L^{p/(p-1)}(\Omega)}T^{(p-1)/p}\\
		\leq&2\max\{\mathrm{meas}(\Omega),1\}(T+1)+ 2\gamma_2(M_1+1)(T+1).
	\end{aligned}
	$$
	Note that $\{\nabla v_m^k\}_{k=1}^{\infty}$ is the $W^{1, p}(Q_T)$-gradient generating sequence of $\nu^m$. We derive
	\begin{align}
		\left\| \nu^m\right\|_{L^1(Q_T;(\mathcal{E}_0^p(\mathbb{R}^n))^{\prime})}&=\iint_{Q_T}\sup_{\left\|f\right\|_{\mathcal{E}^p(\mathbb{R}^n)\leq 1}}\left|\langle \nu_{x,t}^m,f\rangle\right|dxdt\notag\\
		\leq&\iint_{Q_T}\int_{\mathbb{R}^n}(1+\left|\xi\right|^p)d\nu_{x,t}^m(\xi)dxdt\notag\\
		=&\mathrm{meas}(\Omega)T+\iint_{Q_T}\langle \nu_{x,t}^m,\left|\mathrm{id}\right|^p\rangle dxdt\notag\\
		=&\mathrm{meas}(\Omega)T+\lim_{k\rightarrow\infty}\iint_{Q_T}\left|\nabla v_m^k\right|^p dxdt\notag\\
		=&\mathrm{meas}(\Omega)T+\lim_{k\rightarrow\infty}\sum_{j=1}^{m}\int_{(j-1)\frac{T}{m}}^{j\frac{T}{m}}\int_{\Omega}\left|\nabla u_m^{j,k}\right|^p dxdt\notag\\
		\leq&\mathrm{meas}(\Omega)T+\sup_{k\geq 1}\sup_{1\leq j\leq m}\left\|\nabla u_m^{j,k}\right\|_{L^p(\Omega)}^p T.\notag
	\end{align}
	Inequality (\ref{boundness-for-|nu^m|_L^1-and-<nu^m,q>}) follows from the above inequality and (\ref{iteration-inequlity-for-the-first-term-of-E_m*-with-u_m^jk}).
	
	Now we prove (\ref{boundness-for-du_m/dt-and-(u_m-v_m)}). By (\ref{boundness-for-m/4T-sum-(u_m^j-u_m^j-1)^2}), we see that
	$$
	\begin{aligned}
		\left\|\frac{\partial u_m}{\partial t}\right\|_{L^2(Q_T)}^2=\frac{m}{T}\sum_{j=1}^{m}\int_{\Omega}(u_m^j-u_m^{j-1})^2dx\leq M_2,
	\end{aligned}
	$$
	$$
	\begin{aligned}
		\left\|u_m-v_m\right\|_{L^2(Q_T)}^2\leq\frac{T}{m}\sum_{j=1}^{m}\int_{\Omega}(u_m^j-u_m^{j-1})^2dx\leq \frac{T^2}{m^2}M_2.
	\end{aligned}
	$$ 
	Furthermore, applying Lemma \ref{firstlemma} and (\ref{boundness-of-u_m^j-for-W1pnorm-and-L^inftynorm}), it is straightforward to see that inequalities (\ref{boundness-for-v_m}) and (\ref{boundness-for-u_m}) hold, thus completing the proof.     
\end{proof}
Lemma \ref{secondlemma} implies that there exists a subsequence of $\{v_m\}_{m=1}^{\infty}$, denoted by itself, and a function $v_{w}\in L^{\infty}(0,T;W_{l,d}^{1,p}(\Omega))\cap L^{\infty}(Q_T)$, such that as $m\rightarrow\infty$
\begin{align}
	&v_m\rightharpoonup v_{w}~\text{weakly in}~L^p(Q_T),\notag\\
	&\nabla v_m\rightharpoonup \nabla v_{w}~\text{weakly in}~L^p(Q_T).\label{convergence-of-v_m}
\end{align}
Owing to (\ref{boundness-for-u_m}), (\ref{boundness-for-du_m/dt-and-(u_m-v_m)}), and then utilizing Rellich theorem, we conclude that there exists a subsequence of $\{u_m\}_{m=1}^{\infty}$, still denoted by itself, and a function $u_w\in L^{\infty}(0,T;W_{l,d}^{1,p}(\Omega))\cap L^{\infty}(Q_T)$, such that as $m\rightarrow\infty$
\begin{align}
	&u_m \rightarrow u_{w}~\text{strongly in}~L^p(Q_T),\notag\\
	&\nabla u_m\rightharpoonup \nabla u_{w}~\text{weakly in}~ L^p(Q_T),\notag\\
	&\frac{\partial u_m}{\partial t}\rightharpoonup\frac{\partial u_{w}}{\partial t}~\text{weakly in}~L^2(Q_T).\label{convergence-of-u_m}
\end{align}
On the other hand, we apply Remark \ref{firstremark} following Lemma \ref{WangChunpeng-YoungMeasure} on $\nu^m$ to extract a subsequence, still indexed by $m\rightarrow\infty$, along which the $\nu^m$ converges to a $W^{1,p}(Q_T)$-gradient Young measure $\nu^w$. More precisely,
\begin{equation}
	\langle\nu^m,\psi\rangle\rightharpoonup\langle\nu^{w},\psi\rangle\label{convergenceformeasure}
\end{equation}
weakly$^\ast$ in $L^{\infty}(Q_T)$ for $\psi\in C_0(\mathbb{R}^n)$, weakly in $L^1(Q_T)$ for $\psi\in \mathcal{E}_0^1(\mathbb{R}^n)$, and weakly in the biting sense for $\psi\in \mathcal{E}_0^p(\mathbb{R}^n)$. In particular, since $\vec{q}\in\mathcal{E}_0^1(\mathbb{R}^n)$ and $\vec{q}\cdot\mathrm{id}\in\mathcal{E}_0^p(\mathbb{R}^n)$, we have from (\ref{gradientforum}) and (\ref{boundness-for-du_m/dt-and-(u_m-v_m)}) that
\begin{align}
	&\langle\nu^m,q\rangle\rightharpoonup\langle\nu^{w},q\rangle~\text{weakly in}~L^{p/(p-1)}(Q_T),\notag\\
	&\langle\nu^m,\mathrm{id}\rangle\rightharpoonup\langle\nu^{w},\mathrm{id}\rangle~\text{weakly in}~L^p(Q_T),\notag\\
	&\langle\nu^m,q\cdot\mathrm{id}\rangle\rightharpoonup\langle\nu^{w},q\cdot\mathrm{id}\rangle~\text{in the biting sense}.\label{particular-convergence-of-nu^m}
\end{align}

Next, we show that the problem ($\mathcal{P}_{w}$) admits a Young measure solution $u_w\in W(0,T)$.
\begin{lem}\label{firsttheorem}
	Assume that $f\in W^{1,p}(\Omega)$ and \emph{(\ref{assumption-for-u0})} hold, then problem \emph{($\mathcal{P}_{w}$)} possesses a Young measure solution $u_w$ such that
	\begin{equation}
		l\leq u_{w}(x,t)\leq d,\quad \emph{a.e.}~(x,t)\in Q_T,\label{l<u_w<d}
	\end{equation} 
	and
	\begin{equation}
		\left\| u_{w}\right\|_{L^{\infty}(0,T;W^{1,p}(\Omega))}\leq M,\quad\left\|\frac{\partial u_{w}}{\partial t}\right\|_{L^2(Q_T)}\leq M,\label{boundness-for-u_w-L^infty(0,T;W^1,p)-du_w/dt-L^2} 
	\end{equation}
	where $M>0$ only depends on $\Omega$, $T$, $M_1$, $M_2$ and $\gamma_2$.
\end{lem}
\begin{proof}
	Noting (\ref{boundness-for-du_m/dt-and-(u_m-v_m)}) implies
	$$
	\lim\limits_{m\rightarrow\infty}\left\|u_m-v_m\right\|_{L^2(Q_T)}=0,
	$$
	one can subsequently get that
	\begin{equation}
		v_m\rightarrow u_w~\text{strongly in}~L^2(Q_T),~\text{as}~m\rightarrow\infty,\label{convergence-of-v_m-strongly}
	\end{equation}
	\begin{equation}
		u_{w}(x,t)=v_{w}(x,t),\quad\text{a.e.}~ (x,t)\in Q_T.\label{u_w=v_w}
	\end{equation}
	According to the definition of $w_m$, and the condition that $w\in C(0,T;L^2(\Omega))$, we obtain
	\begin{equation}
		w_m\rightarrow w~\text{strongly in}~L^1(Q_T),~\text{as}~m\rightarrow\infty.\label{convergence-of-w_m}
	\end{equation}
	Combining (\ref{l<w<d}), (\ref{convergence-of-v_m-strongly}), (\ref{convergence-of-w_m}) and using Lebesgue’s dominated convergence theorem, we have
	\begin{equation}
		\frac{v_m-f}{w_m^2}\rightarrow\frac{u_{w}-f}{w^2}~\text{strongly in}~L^1(Q_T),~\text{as}~m\rightarrow\infty.\label{convergence-of-(v_m-u_0)/w_m^2}
	\end{equation}
	
	Now we are ready to verify that $u_w$ with $\nu^w$ is a Young measure solution to the problem ($\mathcal{P}_w$). First, letting $m\rightarrow\infty$ in (\ref{equilibrium-equation-of-u_m-and-nu^m}) with (\ref{convergence-of-u_m}), (\ref{particular-convergence-of-nu^m}) and (\ref{convergence-of-(v_m-u_0)/w_m^2}) leads to (\ref{equilibrium-equation-for-uw}) for any $\eta\in C^{\infty}(\overline{Q_T})$. Second, (\ref{gradient-of-uw}) follows from (\ref{gradientforum}), (\ref{convergence-of-v_m}), (\ref{particular-convergence-of-nu^m}) and (\ref{u_w=v_w}). Third, (\ref{support-of-uw}) can be deduced from (\ref{suppform}) and (\ref{convergenceformeasure}), while (\ref{boundness-for-du_m/dt-and-(u_m-v_m)}) guarantees that (\ref{trace-of-uw}) holds in the sense of trace.
	
	Finally, we prove the independence property (\ref{independence-for-uw}). Recall that $\{u_m^{j,k}\}_{k=1}^{\infty}$ serves as a minimizing sequence to the variational principle $E_m^{**}(\cdot;u_m^{j-1})$ converging to $u_m^j$ weakly in $W^{1,p}(\Omega)$, strongly in $L^p(\Omega)$ and weakly$^\ast$ in $L^{\infty}(\Omega)$. In addition, $\{\nabla u_m^{j,k}\}_{k=1}^{\infty}$ generates the $W^{1,p}(\Omega)$-gradient Young measure $\nu^{m,j}$. For all $\zeta\in W^{1,p}(\Omega)$, we see that
	$$
	\begin{aligned}
		&\lim\limits_{k\rightarrow\infty}\int_{\Omega}\left(\alpha(x)q^{**}(\nabla u_m^{j,k})\cdot\nabla\zeta+\frac{m}{T}(u_m^{j,k}-u_m^{j-1})\zeta+\frac{u_m^{j,k}-f}{(w_m^j)^2}\zeta\right)dx\\
		=&\int_{\Omega}\left(\alpha(x)\langle \nu^{m,j},q^{**}\rangle\cdot\nabla\zeta+\frac{m}{T}(u_m^j-u_m^{j-1})\zeta+\frac{u_m^j-f}{(w_m^j)^2}\zeta\right)dx,
	\end{aligned}
	$$
	and
	$$
	\begin{aligned}
		&\lim_{k\rightarrow\infty}\left|\int_{\Omega}\alpha(x)\left(q^{**}(\nabla u_m^{j,k})-\langle \nu^{m,j},q^{**}\rangle\right)\cdot\nabla\left( u_m^{j,k}-u_m^j\right)dx \right|\\
		=&\lim_{k\rightarrow\infty}\left|\int_{\Omega}\frac{m}{T}(u_m^j-u_m^{j,k})^2 dx+\int_{\Omega}\frac{(u_m^{j,k}-u_m^j)^2}{(w_m^j)^2}dx\right|\\
		\leq& \lim_{k\rightarrow\infty}\left(\frac{m}{T}+\frac{1}{l^2}\right)\left\|u_m^{j,k}-u_m^j\right\|^2_{L^2(\Omega)}=0.
	\end{aligned}
	$$
	Since $q^{**}(\nabla u_m^{j,k})\cdot\nabla u_m^{j,k}$ converges weakly to $\langle\nu^{m,j},q^{**}\cdot\mathrm{id}\rangle$ in $L^1(\Omega)$, $q^{**}(\nabla u_m^{j,k})$ converges weakly to $\langle\nu^{m,j},q^{**}\rangle$ in $L^{p/(p-1)}(\Omega)$ and $\nabla u_m^{j,k}$ converges weakly to $\langle\nu^{m,j},\mathrm{id}\rangle$ in $L^{p}(\Omega)$ as $k\rightarrow\infty$, it can be concluded that
	$$
	\langle\nu_x^{m,j},q^{**}\cdot\mathrm{id}\rangle=\langle\nu_x^{m,j},q^{**}\rangle\cdot\langle\nu_x^{m,j},\mathrm{id}\rangle,\quad\text{a.e.}~x\in\Omega.
	$$
	Thus (\ref{supportset-of-nu_m^j}) implies that
	$$
	\langle\nu_x^{m,j},q\cdot\mathrm{id}\rangle=\langle\nu_x^{m,j},q\rangle\cdot\langle\nu_x^{m,j},\mathrm{id}\rangle,\quad\text{a.e.}~x\in\Omega.
	$$
	By the definition of $\nu^m$, we have
	$$
	\langle\nu_{x,t}^{m},q\cdot\mathrm{id}\rangle=\langle\nu_{x,t}^{m},q\rangle\cdot\langle\nu_{x,t}^{m},\mathrm{id}\rangle,\quad\text{a.e.}~(x,t)\in Q_T.
	$$
	Choosing $v_m\psi$ for $\psi\in C_0^{\infty}(Q_T)$ as a test function in (\ref{equilibrium-equation-for-uw}), (\ref{equilibrium-equation-of-u_m-and-nu^m}), we notice
	\begin{align}
		&\left|\iint_{Q_T}\alpha\left(\langle\nu^m,q\rangle-\langle\nu^{w},q\rangle\right)\cdot\langle\nu^m,\mathrm{id}\rangle\psi dxdt\right|\notag\\
		\leq&\left|\iint_{Q_T}\alpha\left(\langle\nu^m,q\rangle-\langle\nu^{w},q\rangle\right)\cdot\nabla(v_m\psi)dxdt\right|\notag\\
		&+\left|\iint_{Q_T}v_m\alpha\left(\langle\nu^m,q\rangle-\langle\nu^{w},q\rangle\right)\cdot \nabla\psi dxdt\right|\notag\\
		\leq&\left|\iint_{Q_T}u_{w}\alpha\left(\langle\nu^m,q\rangle-\langle\nu^{w},q\rangle\right)\cdot\nabla\psi dxdt\right|\notag\\
		&+\left|\iint_{Q_T}\alpha(v_m-u_{w})\left(\langle\nu^m,q\rangle-\langle\nu^{w},q\rangle\right)\cdot\nabla\psi dxdt\right|\notag\\
		&+\left|\iint_{Q_T}\left(\frac{\partial u_m}{\partial t}-\frac{\partial u_{w}}{\partial t}\right)u_{w}\psi dxdt\right|\notag\\
		&+\left|\iint_{Q_T}\left(\frac{\partial u_m}{\partial t}-\frac{\partial u_{w}}{\partial t}\right)(v_m-u_{w})\psi dxdt\right|\notag\\
		&+\left|\iint_{Q_T}\left(\frac{v_m-f}{w_m^2}-\frac{u_{w}-f}{w^2}\right)v_m\psi dxdt\right|.\notag
	\end{align}
	Combining this with (\ref{convergence-of-u_m}), (\ref{particular-convergence-of-nu^m}), (\ref{convergence-of-v_m-strongly}), (\ref{convergence-of-(v_m-u_0)/w_m^2}),
	and applying H\"{o}lder inequality, then we have
	$$
	\left|\iint_{Q_T}\alpha(x)\left(\langle\nu^m,q\rangle-\langle\nu^{w},q\rangle\right)\cdot\langle\nu^m,\mathrm{id}\rangle\psi dxdt\right|\rightarrow 0,~\text{as}~m\rightarrow\infty.
	$$
	Thus,
	$$
	\begin{aligned}
		&\left|\iint_{Q_T}\left(\alpha(x)\langle\nu^m,q\rangle\cdot\langle\nu^m,\mathrm{id}\rangle-\alpha\langle\nu^{w},q\rangle\cdot\langle\nu^{w},\mathrm{id}\rangle\right)\psi dxdt\right|\notag\\
		\leq&\left|\iint_{Q_T}\alpha(x)\left(\langle\nu^m,\vec{q}\rangle-\langle\nu^{w},q\rangle\right)\cdot\langle\nu^m,\mathrm{id}\rangle\psi dxdt\right|\notag\\
		&+\left|\iint_{Q_T}\alpha(x)\langle\nu^{w},q\rangle\cdot\left(\langle\nu^m,\mathrm{id}\rangle-\langle\nu^{w},\mathrm{id}\rangle\right)\psi dxdt\right|\rightarrow 0,~\text{as}~m\rightarrow\infty.
	\end{aligned}
	$$
	So $\langle\nu^m,q\rangle\cdot\langle\nu^m,\mathrm{id}\rangle$ converges weakly to $\langle\nu^{w},q\rangle\cdot\langle\nu^{w},\mathrm{id}\rangle$ in $L^1(Q_T)$. Considering that $\langle\nu^m,q\cdot\mathrm{id}\rangle$ converges to $\langle\nu,q\cdot\mathrm{id}\rangle$ in the biting sense, there exists a decreasing sequence of subsets $E_{j+1}\subseteq E_j$ of $Q_T$ with $\lim_{j\rightarrow\infty}\mathrm{meas}(E_j)=0$ such that
	$$
	\langle\nu^{w},q\cdot\mathrm{id}\rangle=\langle\nu^{w},q\rangle\cdot\langle\nu^{w},\mathrm{id}\rangle,\quad\text{a.e.}~(x,t)\in Q_T\backslash E_j,~j=1,2,\cdots.
	$$
	Since $\langle\nu^{w},q\cdot\mathrm{id}\rangle\in L^1(Q_T)$, we get that
	$$
	\langle\nu^{w},q\cdot\mathrm{id}\rangle=\langle\nu^{w},q\rangle\cdot\langle\nu^{w},\mathrm{id}\rangle,\quad\text{a.e.}~(x,t)\in Q_T.
	$$
	Therefore, $u_w$ is the desired Young measure solution to the problem ($\mathcal{P}_w$). Furthermore, (\ref{l<u_w<d}) and (\ref{boundness-for-u_w-L^infty(0,T;W^1,p)-du_w/dt-L^2}) easily follow from (\ref{boundness-for-u_m}) and (\ref{boundness-for-du_m/dt-and-(u_m-v_m)}). The proof is complete.
\end{proof}
\section{Proof of Theorem \ref{thirdtheorem}}\label{proof-of-main-thm}
In this section, we will prove the existence of Young measure solutions to the problem (\ref{P1})--(\ref{P3}). We begin by presenting the uniqueness result for problem ($\mathcal{P}_w$), which is crucial for our subsequent analysis. 

It is important to note that uniqueness only pertains to the function $u_w$, while the gradient Young measure $\nu^w=(\nu^w_{x,t})_{(x,t)\in Q_T}$ is in general not unique. Here, the independence property (\ref{independence-for-uw}) plays a significant role in the following proof.
\begin{lem}\label{secondtheorem}
	Assume that $u_w$ with $\nu^w$ is a Young measure solution of problem \emph{($\mathcal{P}_{w}$)}, then the function $u_w$ is unique.
\end{lem}
\begin{proof}
	Suppose $u_w,~z_w\in W(0,T)\cap L^{\infty}(Q_T)$ are two Young measure solutions to problem ($\mathcal{P}_w$). Let $\nu^w$ and $\mu^w$ be the $W^{1,p}(Q_T)$-gradient Young measure with respect to $u_w$ and $z_w$, respectively. For any $s\in [0,T]$, choosing
	$$
	\eta(x,t)=\left(u_{w}(x,t)-z_{w}(x,t)\right)\chi_{[0,s]}(t),\quad (x,t)\in Q_T,
	$$
	in (\ref{equilibrium-equation-for-uw}), one gets that
	\begin{align}
		\iint_{Q_s}&\alpha(x)\langle \nu^w,q\rangle \cdot \nabla (u_w-z_w)dxdt+\iint_{Q_s}\frac{\partial u_w}{\partial t} (u_w-z_w)dxdt\notag\\
		&+\iint_{Q_s}\frac{u_w-f}{w^2}(u_w-z_w)dxdt=0,\notag
	\end{align}
	and
	\begin{align}
		\iint_{Q_s}&\alpha(x)\langle \mu^w, q\rangle \cdot \nabla (u_w-z_w)dxdt+\iint_{Q_s}\frac{\partial z_w}{\partial t} (u_w-z_w)dxdt\notag\\
		&+\iint_{Q_s}\frac{z_w-f}{w^2}(u_w-z_w)dxdt=0.\notag
	\end{align}
	Combine these two equalities to get that
	\begin{align}
		&\int_{\Omega}\left(u_w(x,s)-z_w(x,s)\right)^2dx
		+2\iint_{Q_s}\frac{(u_w-z_w)^2}{w^2}dxdt\notag\\
		=&
		-2\iint_{Q_s}\alpha(x)\left(\langle\nu^w, q\rangle-\langle\mu^w, q\rangle\right)\cdot\nabla (u_w-z_w)dxdt.\label{equality-for-(u_w,z_w)-and-(nu^w,mu^w)}	
	\end{align}
	We estimate the term on the right side of (\ref{equality-for-(u_w,z_w)-and-(nu^w,mu^w)}). From (\ref{independence-for-uw}) and (\ref{support-of-uw}), it follows that
 \begin{align}
     &\iint_{Q_s}\alpha(x)\left(\langle\nu^w, q\rangle-\langle\mu^w, q\rangle\right)\cdot\nabla (u_w-z_w)dxdt\notag\\
		=&\iint_{Q_s}\alpha(x)\left(\langle\nu^w, q\rangle-\langle\nu^w, q\rangle\right)\cdot \left(\langle \nu^w, \mathrm{id}\rangle-\langle \mu^w, \mathrm{id}\rangle\right)dxdt\notag\\
		=&\iint_{Q_s} \alpha(x)\langle\nu^w, \vec{q}\rangle\cdot\langle\nu^w, \mathrm{id}\rangle-\iint_{Q_s}\alpha(x)\langle \nu^w, q\rangle\cdot\langle \mu^w, \mathrm{id}\rangle dxdt\notag\\
		&-\iint_{Q_s}\alpha(x)\langle \mu^w, q\rangle\cdot\langle \nu^w, \mathrm{id}\rangle dxdt+\iint_{Q_s}\alpha(x)\langle \mu^w, q\rangle\cdot\langle \mu^w, \mathrm{id}\rangle dxdt\notag\\
		=&\iint_{Q_s} \alpha(x)\left(\langle \nu^w, q\cdot\mathrm{id}\rangle-\langle\nu^w, q\rangle\cdot\langle \mu^w, \mathrm{id}\rangle-\langle \mu^w, q\rangle\cdot\langle \nu^w, \mathrm{id}\rangle \notag\right. \\
		&\left. +\langle \mu^w, q\cdot\mathrm{id}\rangle\right) dxdt\notag\\
		=&\iint_{Q_s}\alpha(x)\int_{\mathbb{R}^n}\int_{\mathbb{R}^n}\left(q(\lambda)-q(\sigma)\right)\cdot(\lambda-\sigma)d\nu^w(\lambda)d\mu^w(\sigma)dxdt\notag\\
  =&\iint_{Q_s}\alpha(x)\int_{\mathbb{R}^n}\int_{\mathbb{R}^n}\left(q^{**}(\lambda)-q^{**}(\sigma)\right)\cdot(\lambda-\sigma)d\nu^w(\lambda)d\mu^w(\sigma)dxdt\notag\\
		=&\iint_{Q_s}\int_{\mathbb{R}^n}\int_{\mathbb{R}^n}\alpha(x)\left(\nabla\varphi^{**}(\lambda)-\nabla\varphi^{**}(\sigma)\right)\cdot(\lambda-\sigma)d\nu^w(\lambda)\notag\\
		&d\mu^w(\sigma)dxdt,
 \end{align}
 which, together with the convexity of $\varphi^{**}$, implies
	\begin{equation}
		\iint_{Q_s}\alpha(x)\left(\langle \nu^w, q\rangle-\langle \mu^w, q\rangle\right)\cdot\nabla (u_w-z_w)dxdt\geq 0.\label{inequality-for-(nu^w-mu^w)-and-(u_w-z_w)}
	\end{equation}
	Substituting (\ref{inequality-for-(nu^w-mu^w)-and-(u_w-z_w)}) into (\ref{equality-for-(u_w,z_w)-and-(nu^w,mu^w)}), we obtain $u_w(x,t)=z_w(x,t)$, a.e. $(x,t)\in Q_T$, and this shows uniqueness.
\end{proof}

Now we present the proof of our main result using Schauder's fixed-point theorem.
\begin{proof}[Proof of Theorem \ref{thirdtheorem}]
To proceed, we introduce the subset $W_0$ of $W(0,T)$, defined by
		$$
		W_0=\left\{
		\begin{aligned}
			&u\in W(0,T),\left\| u\right\|_{L^{\infty}(0,T;W^{1,p}(\Omega))}\leq M,\\
			&\left\| \frac{\partial u}{\partial t} \right\|_{L^2(Q_T)}\leq M,~l\leq u\leq d,~\text{a.e.}~(x,t)\in Q_T,\\
			& \text{and}~u(x,0)=f(x) \text{ in the sense of trace}
		\end{aligned}
		\right\}
		$$
		Then, we find that $W_0$ is a nonempty, convex, and compact subset of $W(0,T)$. Consider a mapping
		\begin{align}
			\mathcal{T}:&W_0\rightarrow W_0\notag\\
			&w\mapsto u_{w}\notag
		\end{align}
		In order to use the Schauder's fixed-point theorem on $\mathcal{T}$, we only need to prove that the mapping $\mathcal{T}:w\rightarrow u_w$ is continuous from $W_0$ into $W_0$. Let $w_k$ be a sequence that converges to some $w$ in $W_0$, and let $u_{w_k}=u_k$ be the Young measure solution of problem ($\mathcal{P}_{w_k}$) with the initial data $f$ with respect to the $W^{1,p}(Q_T)$-gradient Young measures $\nu^{w_k}=\nu^k$; consequently, it follows that
		\begin{equation}
			\iint_{Q_T}\left(\alpha(x)\langle\nu^k,q\rangle \cdot\nabla\eta+\frac{\partial u_k}{\partial t}\eta+\frac{u_k-f}{w_k^2}\eta\right)dxdt=0,\label{equilibrium-for-uk}
		\end{equation}
		for any $\eta\in C^{\infty}(\overline{Q_T})$, and
		\begin{equation}
			\nabla u_k(x,t)=\langle\nu^k_{x,t},\mathrm{id}\rangle,\quad\text{a.e.}~(x,t)\in Q_T,\label{gradient-of-uk}
		\end{equation}
		\begin{equation}
			\langle\nu^k_{x,t},q\cdot \mathrm{id}\rangle=\langle\nu^k_{x,t}, q\rangle\cdot\langle\nu^k_{x,t}, \mathrm{id}\rangle,\quad\text{a.e.}~(x,t)\in Q_T,\label{independence-for-nu^k}
		\end{equation}
		\begin{equation}
			\operatorname{supp}\nu^k_{x,t}\subseteq\left\{\xi\in \mathbb{R}^n: \varphi(\xi)=\varphi^{**}(\xi)\right\},\quad\text{a.e.}~(x,t)\in Q_T.\label{support-of-nu^k}
		\end{equation}
		Also, $u_k(x,0)=u_0(x)$, $x\in\Omega$ in the sense of trace. Furthermore, from the proof of Lemma \ref{firsttheorem}, we see that $u_k\in W_0$, and $\nu^k$ satisfies the estimates
		\begin{equation}
			\left\|\nu^k\right\|_{L^1(Q_T;(\mathcal{E}_0^p(\mathbb{R}^n))^{\prime})}\leq M,\quad \left\|\langle\nu^k,q \rangle\right\|_{L^{p/(p-1)}(Q_T)}\leq M,\quad k\geq 1.\label{boundness-for-nu^k}
		\end{equation}
		Subsequently, we will demonstrate that $\mathcal{T}(w_k):=u_k$ converges to $\mathcal{T}(w):=u_w$ in $W_0$. From the above estimates and classical theorems of compact inclusion in Sobolev spaces, we can extract from $w_k$, respectively from $u_k$, a subsequence (labeled $w_k$, respectively $u_k$) such that
		\begin{align}
			&u_k\rightharpoonup u~\text{weakly}^\ast~\text{in}~L^{\infty}\left(0,T;W^{1,p}(\Omega)\right),\notag\\
			&w_k\rightharpoonup w~\text{weakly}^\ast~\text{in}~L^{\infty}\left(0,T;W^{1,p}(\Omega)\right),\notag\\
			&u_k\rightarrow u~\text{strongly in}~L^p(Q_T)~\text{and a.e. on}~Q_T,\notag\\
			&w_k\rightarrow w~\text{strongly in}~L^2(Q_T)~\text{and a.e. on}~Q_T,\notag\\
			&\frac{u_k-f}{w_k^2}\rightarrow\frac{u-f}{w^2}~\text{strongly in}~L^1(Q_T)~\text{and a.e. on}~Q_T,\notag\\
			&\nabla u_k\rightharpoonup\nabla u~\text{weakly in}~L^p(Q_T),\notag\\
			&\frac{\partial u_k}{\partial t}\rightharpoonup\frac{\partial u}{\partial t}~\text{weakly in}~L^2(Q_T).\label{all-convergence}
		\end{align}
		In addition, applying Lemma \ref{WangChunpeng-YoungMeasure} with the uniform estimates (\ref{boundness-for-nu^k}), there exists a subsequence (not relabeled) of the $(\nu^k)_{k\geq 1}$ and a $W^{1,p}(Q_T)$-gradient Young measure $\nu=(\nu_{x,t})_{(x,t)\in Q_T}$, such that
		\begin{align}
			&\langle\nu^k,\mathrm{id}\rangle\rightharpoonup\langle\nu,\mathrm{id}\rangle~\text{weakly in}~L^p(Q_T),\notag\\
			&\langle\nu^k,q\rangle\rightharpoonup\langle\nu,q\rangle~\text{weakly in}~L^{p/(p-1)}(Q_T),\notag\\
			&\langle\nu^k,q\cdot\mathrm{id}\rangle\rightharpoonup\langle\nu,q\cdot\mathrm{id}\rangle~\text{in the biting sense}.\label{all-convergence-for-measure}
		\end{align}
		From the above convergences we can pass to the limit in ($\mathcal{P}_{w_k}$) and we now show that $u=\mathcal{T}(w)$. First, letting $k\rightarrow\infty$ in (\ref{equilibrium-for-uk}) with (\ref{all-convergence}) and (\ref{all-convergence-for-measure}), we obtain
		\begin{equation}
			\iint_{Q_T}\left(\alpha(x)\langle\nu,q\rangle \cdot\nabla\eta+\frac{\partial u}{\partial t}\eta+\frac{u-f}{w^2}\eta\right)dxdt=0,\quad\forall\eta\in C^{\infty}(\overline{Q_T}).\label{equilibrium-for-u_w-and-w}
		\end{equation}
		Second, recall that $\nabla u_k=\langle\nu^k,\mathrm{id}\rangle$ and that it converges weakly to $\nabla u$ in $L^p(Q_T)$, and (\ref{all-convergence-for-measure}) yields
		$$
		\nabla u(x,t)=\langle\nu_{x,t},\mathrm{id}\rangle,\quad\text{a.e.}~(x,t)\in Q_T.
		$$
		Third, by the similar methods as in Theorem \ref{firsttheorem} for proving the independence property (\ref{independence-for-uw}), we can deduce from (\ref{equilibrium-equation-for-u}) and (\ref{equilibrium-for-uk}) that for any $\psi\in C_0^{\infty}(Q_T)$,
		\begin{align}
			&\left|\iint_{Q_T}\alpha\left(\langle\nu^k,q\rangle-\langle\nu,q\rangle\right)\cdot\langle\nu^k,\mathrm{id}\rangle\psi dxdt\right|\notag\\
			\leq&\left|\iint_{Q_T}\alpha\left(\langle\nu^k,q\rangle-\langle\nu,q\rangle\right)\cdot\nabla(u_k\psi)dxdt\right|\notag\\
			&+\left|\iint_{Q_T}u_k\alpha\left(\langle\nu^k,q\rangle-\langle\nu,q\rangle\right)\cdot \nabla\psi dxdt\right|\notag\\
			\leq &\left|\iint_{Q_T}\left(\frac{\partial u_k}{\partial t}-\frac{\partial u}{\partial t}\right)u\psi dxdt\right|\notag\\
			&+\left|\iint_{Q_T}\left(\frac{\partial u_k}{\partial t}-\frac{\partial u}{\partial t}\right)(u_k-u)\psi dxdt\right|\notag\\
			&+\left|\iint_{Q_T}\left(\frac{u_k-f}{w_k^2}-\frac{u-f}{w^2}\right)u_k\psi dxdt\right|\notag\\
			&+\left|\iint_{Q_T}u\alpha\left(\langle\nu^k,q\rangle-\langle\nu,q\rangle\right)\cdot\nabla\psi dxdt\right|\notag\\
			&+\left|\iint_{Q_T}\alpha(u_k-u)\left(\langle\nu^k,q\rangle-\langle\nu,q\rangle\right)\cdot\nabla\psi dxdt\right|\rightarrow 0,~\text{as}~k\rightarrow\infty.
		\end{align}
		Note that
		$$
		\begin{aligned}
			&\left|\iint_{Q_T}\alpha(x)\left(\langle\nu^k,q\rangle\cdot\langle\nu^k,\mathrm{id}\rangle-\langle\nu,q\rangle\cdot\langle\nu,\mathrm{id}\rangle\right)\psi dxdt\right|\notag\\
			\leq&\left|\iint_{Q_T}\alpha(x)\left(\langle\nu^k,q\rangle-\langle\nu,q\rangle\right)\cdot\langle\nu^k,\mathrm{id}\rangle\psi dxdt\right|\notag\\
			&+\left|\iint_{Q_T}\alpha(x)\langle\nu,q\rangle\cdot\left(\langle\nu^k,\mathrm{id}\rangle-\langle\nu,\mathrm{id}\rangle\right)\psi dxdt\right|\rightarrow 0,~\text{as}~k\rightarrow\infty,
		\end{aligned}
		$$
		which, together with (\ref{independence-for-nu^k}), (\ref{all-convergence-for-measure}), implies
		$$
		\langle\nu_{x,t},q\cdot\mathrm{id}\rangle=\langle\nu_{x,t},q\rangle\cdot\langle\nu_{x,t},\mathrm{id}\rangle,\quad\text{a.e.}~(x,t)\in Q_T\backslash F_j,\quad j=1,2,\cdots,
		$$
		where $\mathrm{meas}(F_j)\rightarrow 0$, as $j\rightarrow\infty$. Since $\langle\nu_{x,t},\vec{q}\cdot\mathrm{id}\rangle\in L^1(Q_T)$, we see that
		$$
		\langle\nu_{x,t},q\cdot\mathrm{id}\rangle=\langle\nu_{x,t},q\rangle\cdot\langle\nu_{x,t},\mathrm{id}\rangle,\quad\text{a.e.}~(x,t)\in Q_T,
		$$
		Finally, (\ref{support-of-u}) follows from (\ref{support-of-nu^k}) and (\ref{all-convergence-for-measure}), while (\ref{all-convergence}) guarantees that (\ref{P3}) holds in the sense of trace. 
		
		Therefore, $u$ with $\nu$ is a Young measure solution of problem ($\mathcal{P}_{w}$) and we have $u=u_w=\mathcal{T}(w)$. Moreover, $u_k$ is unique due to Theorem \ref{secondtheorem}, the whole sequence $u_k=\mathcal{T}(w_k)$ converges to $u=\mathcal{T}(w)$ in $W_0$. Hence $\mathcal{T}$ is continuous. Consequently, thanks to the Schauder fixed point theorem, there exists $w\in W_0$ such that $w=\mathcal{T}(w)=u_w$. Thus, $u$ is a Young measure solution of problem (\ref{P1})--(\ref{P3}). This establishes the existence of Young measure solutions.

    \textbf{Continuous dependence and uniqueness of Young measure solutions}: Suppose $u,v\in W_0$ are two Young measure solutions to the problem (\ref{P1})--(\ref{P2}) corresponding to the initial data $f_{01},f_{02}$, respectively. Let $\nu$ and $\mu$ be the $W^{1,p}(Q_T)$-gradient Young measure with respect to $u$ and $v$, respectively. Similar to the proof of (\ref{equality-for-(u_w,z_w)-and-(nu^w,mu^w)}), apply equation (\ref{equilibrium-equation-for-u}) using $(u-v)\chi_{[0,s]}$, $s\in [0,T]$ as the test function to obtain
	\begin{align}
		&\int_{\Omega}\left(u(x,s)-v(x,s)\right)^2 dx-\int_{\Omega}\left(f_{01}-f_{02}\right)^2 dx\notag\\
		=&-2\iint_{Q_s}\left(\frac{u-f_{01}}{u^2}-\frac{v-f_{02}}{v^2}\right)(u-v) dxdt
		\notag\\
		&-2\iint_{Q_s}\alpha(x)\left(\langle\nu, q\rangle-\langle\mu,q\rangle\right)\cdot\nabla(u-v)dxdt,\label{equality-for-(u,v)-and-(nu,mu)}
	\end{align}
 Now, we estimate each term on the right side of (\ref{equality-for-(u,v)-and-(nu,mu)}). Since $u$, $v\in W_0$, using the mean value theorem yields

 \begin{align}
     &-2\iint_{Q_s}\left(\frac{u-f_{01}}{u^2}-\frac{v-f_{02}}{v^2}\right)(u-v) dxdt
		\notag\\
		= & 2\iint_{Q_s}\left(\frac{1}{v}-\frac{1}{u} \right)(u-v) dxdt+2\iint_{Q_s}\left(\frac{f_{01}}{u^2}-\frac{f_{02}}{u^2}\right)(u-v) dxdt\notag\\
		&+2\iint_{Q_s}\left(\frac{f_{02}}{u^2}-\frac{f_{02}}{v^2}\right)(u-v) dxdt\notag\\
		= & 2\iint_{Q_s}\int_{0}^{1}(v+\sigma(u-v))^{-2} d\sigma(u-v)^2 dxdt+2\iint_{Q_s}\frac{f_{01}-f_{02}}{u}\frac{u-v}{u} dxdt\notag\\
		&-4\iint_{Q_s}v_0\int_{0}^{1}(v+\sigma(u-v))^{-3} d\sigma(u-v)^2 dxdt\notag\\
		\leq & 2\iint_{Q_s}\int_{0}^{1}(v+\sigma(u-v))^{-2}(u-v)^2 d\sigma dxdt+\iint_{Q_s}\frac{(f_{01}-f_{02})^2}{u^2} dxdt\notag\\
		&+\iint_{Q_s}\frac{(u-v)^2}{u^2} dxdt\notag\\
  \leq & \frac{3}{l^2}\int_{0}^{s}\int_{\Omega}(u(x,t)-v(x,t))^2 dxdt+\frac{T}{l^2}\int_{\Omega}(f_{01}-f_{02})^2 dx.\notag
 \end{align}
	Combining this, (\ref{inequality-for-(nu^w-mu^w)-and-(u_w-z_w)}) with (\ref{equality-for-(u,v)-and-(nu,mu)}), we obtain
	$$
	\begin{aligned}
		&\int_{\Omega}\left(u(x,s)-v(x,s)\right)^2 dx-\frac{l^2+T}{l^2}\int_{\Omega}(f_{01}(x)-f_{02}(x))^2 dx\\
		\leq& \frac{3}{l^2}\int_{0}^{s}\int_{\Omega}(u(x,t)-v(x,t))^2 dxdt,\quad s\in [0,T],
	\end{aligned}
	$$
	which leads to
 \begin{equation}
		\int_{\Omega}\left(u(x,t)-v(x,t)\right)^2 dx\leq C\int_{\Omega}(f_{01}-f_{02})^2 dx,\quad 0\leq t\leq T,\label{continuous-dependence-of-YoungMeasure-solutions}
    \end{equation}
    by using the Gronwall inequality. Note that when $f_{01}=f_{02}$ is used in (\ref{continuous-dependence-of-YoungMeasure-solutions}), we have $u(x,t)=v(x,t)$, a.e. $(x,t)\in Q_T$, which implies the uniqueness of the Young measure solution to the problem (\ref{P1})--(\ref{P3}). Thus, the proof of Theorem \ref{thirdtheorem} is completed.
\end{proof}
\section{Numerical experiments}
\label{Numerical experiments}
In this section, we test the proposed model for removing multiplicative Gamma noise. 

For the applications \cite{zhang2020,dong2013}, we choose the $\alpha(x)$ in (\ref{P1}) as
$$
\alpha(x)=\left(\frac{f_{\sigma}}{M}\right)^{\beta},\quad x\in\mathbb{R}^2,
$$
where $G_{\sigma}$ is Gaussian function with a variance $\sigma^2$, $f_{\sigma}=G_{\sigma}\ast f$, $M=\max_{x} f_{\sigma}(x)$, and $\beta>0$. The explicit finite
difference \cite{rudin1992} is employed for the numerical implementation. Assume $\tau$ to be the time step size and $h=1$ the space grid size, we first give the following notation:
$$
x_i=ih,\quad y_j=jh,\quad i=0,1,\cdots,I,\quad j=0,1,\cdots,J,
$$
where $I$, $J$ are the pixels numbers of the image,
$$
\begin{aligned}
	&t_n=n\tau,~u_{i,j}^0=f(ih,jh),~u_{i,j}^n=u(ih,jh,t_n),~ n=0,1,2,\cdots,\\
&u_{i,0}^n=u_{i,1}^n,~u_{0,j}^n=u_{1,j}^n,~u_{i,J}^n=u_{i,J-1}^n,~u_{I,j}^n=u_{I-1,j}^n,\\
	&\nabla_{\pm}^x u_{i,j}^n=\pm(u_{i\pm1,j}^n-u_{i,j}^n),~\nabla_{\pm}^y u_{i,j}^n=\pm(u_{i,j\pm 1}^n-u_{i,j}^n),\\
 &m(a,b)=\frac{1}{2}\left(\mathrm{sgn}(a)+\mathrm{sgn}(b)\right) \min(\left|a\right|,\left|b\right|),~\left(f_{\sigma}\right)_{i,j}=f_{\sigma}(ih,jh),\\
 &b_{i,j}^{x,n}=\left((\nabla_{+}^x u_{i,j}^n)^2+(m(\nabla_{+}^y u_{i,j}^n,\nabla_{-}^y u_{i,j}^n))^2\right)^{\frac{1}{2}},\\
 &b_{i,j}^{y,n}=\left((\nabla_{+}^y u_{i,j}^n)^2+(m(\nabla_{+}^x u_{i,j}^n,\nabla_{-}^x u_{i,j}^n))^2\right)^{\frac{1}{2}},\\
 &c_{i,j}^{x,n}=\frac{\nabla_{+}^x u_{i,j}^n}{1+(\nabla_{+}^x u_{i,j}^n)^2+(\nabla_{+}^y u_{i,j}^n)^2}+\delta\frac{\nabla_{+}^x u_{i,j}^n}{(b_{i,j}^{x,n})^{2-p}},\\
	&c_{i,j}^{y,n}=\frac{\nabla_{+}^y u_{i,j}^n}{1+(\nabla_{+}^x u_{i,j}^n)^2+(\nabla_{+}^y u_{i,j}^n)^2}+\delta\frac{\nabla_{+}^y u_{i,j}^n}{(b_{i,j}^{y,n})^{2-p}},
\end{aligned}
$$
and present the discrete explicit scheme for the problem (\ref{P1})--(\ref{P3}):
\begin{align}
	u_{i,j}^{n+1}=&u_{i,j}^n+\tau\left(\nabla_{-}^{x}\left( 
	\frac{\left(f_{\sigma}\right)_{i,j}^\beta}{\max_{i,j}\left(f_{\sigma}\right)_{i,j}^\beta}\cdot c_{i,j}^{x,n} \right)+\nabla_{-}^{y}\left( 
	\frac{\left(f_{\sigma}\right)_{i,j}^\beta}{\max_{i,j}\left(f_{\sigma}\right)_{i,j}^\beta}\cdot c_{i,j}^{y,n} \right)\right)\notag\\
	&+\lambda\tau\frac{u_{i,j}^0-u_{i,j}^n}{(u_{i,j}^n)^2}.\notag
\end{align}
Through the above lines, we can obtain $u_{i,j}^{n+1}$ by $u_{i,j}^{n}$. The restoration quality is measured by the peak signal-to-noise ratio (PSNR) and the mean absolute-deviation error (MAE) values \cite{durand2010}. The stopping criteria of our algorithm is designed through achieving the maximal PSNR.

To verify the effectiveness of our model (\ref{P1})--(\ref{P3}), we test two images including the Synthetic and Nimes, which are distorted by multiplicative Gamma noise with a mean equal to 1 and a variance of $\frac{1}{L}$. The results are compared with AA \cite{aubert2008}, OS \cite{shi2008} and DD \cite{zhou2014} models. 

In Figure \ref{SyntheticL=1}, we show the denoising results for the Synthetic image contaminated by multiplicative Gamma noise with $L=1$. In the visual aspect, the proposed model effectively removes multiplicative noise while preserving more corners and edges than the other three reference models. As illustrated in Figures \ref{AASynthetic12.2963uint8L=1} and \ref{OSSynthetic18.0691L=1}, the restored results of AA and OS model exhibit relatively poor visual quality when dealing with $L=1$ case. AA model’s result suffers from the indistinct edges, and the restored image from OS model tends to be piecewise constant.

Figure \ref{NimesL=4} shows comparisons for the Nimes image, corrupted by multiplicative noise with $L=4$, across each model. The visual quality of the images restored by our model is quite satisfactory. As we can see, our model exhibits comparable performance to the DD model in preserving detailed information, while surpassing it in noise removal. The numerical results indicate that the proposed model achieves a balance between denoising effectiveness and edge preservation.
\begin{figure}[t]
\centering
\subfigure[Original image]{
	\includegraphics[width=0.25\linewidth]{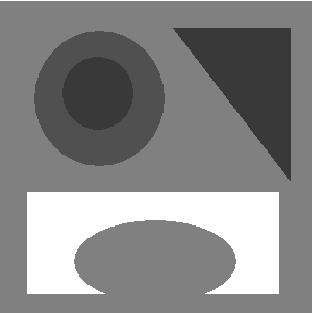}
}
\subfigure[Noise image]{
	\includegraphics[width=0.25\linewidth]{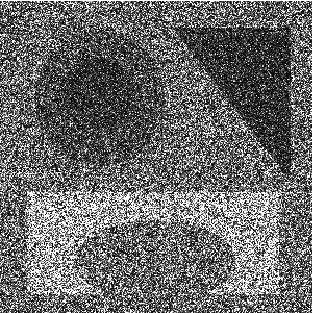}
}
\subfigure[Ours]{
	\includegraphics[width=0.25\linewidth]{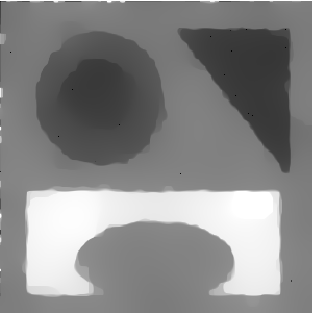}
}
\quad
\subfigure[AA]{
	\includegraphics[width=0.25\linewidth]{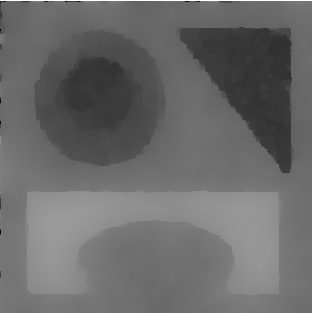}\label{AASynthetic12.2963uint8L=1}
}
\subfigure[OS]{
	\includegraphics[width=0.25\linewidth]{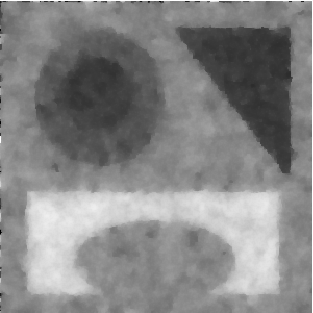}\label{OSSynthetic18.0691L=1}
}
\subfigure[DD]{
	\includegraphics[width=0.25\linewidth]{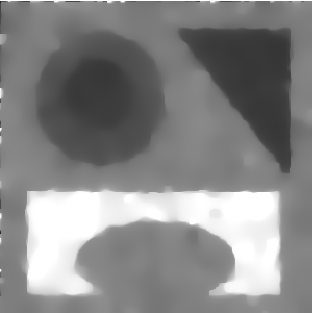}
}
\caption{Experimental results of Ours, AA, OS and DD for the Synthetic image with multiplicative Gamma noise $L=1$}\label{SyntheticL=1}
\end{figure}
\begin{figure}[t]
	\centering
	\subfigure[Original image]{
		\includegraphics[width=0.25\linewidth]{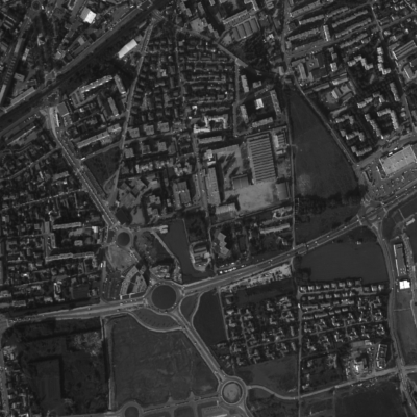}
	}
	\subfigure[Noise image]{
		\includegraphics[width=0.25\linewidth]{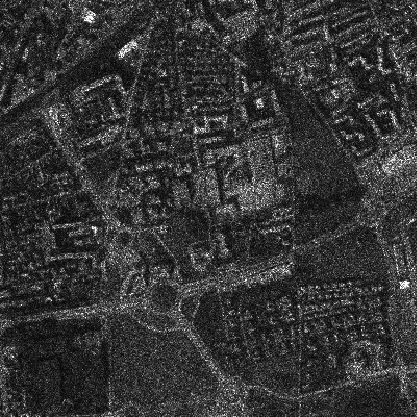}
	}
	\subfigure[Ours]{
		\includegraphics[width=0.25\linewidth]{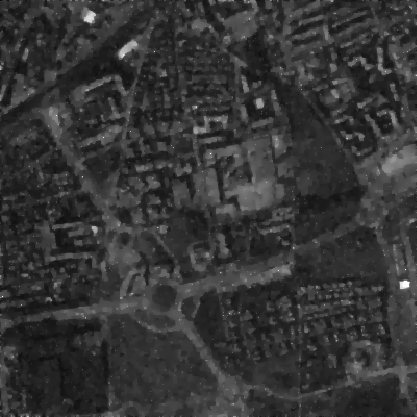}
	}
	\quad
	\subfigure[AA]{
		\includegraphics[width=0.25\linewidth]{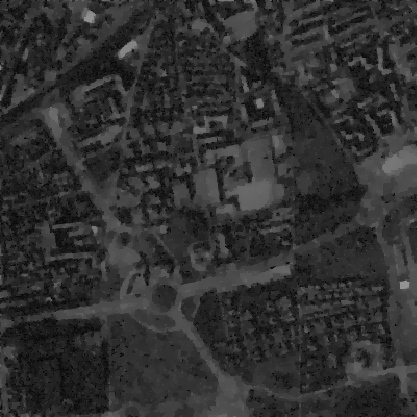}
	}
	\subfigure[OS]{
		\includegraphics[width=0.25\linewidth]{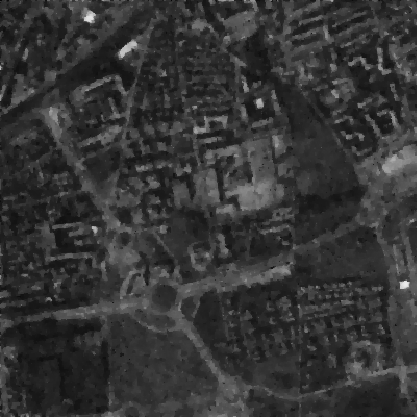}
	}
	\subfigure[DD]{
		\includegraphics[width=0.25\linewidth]{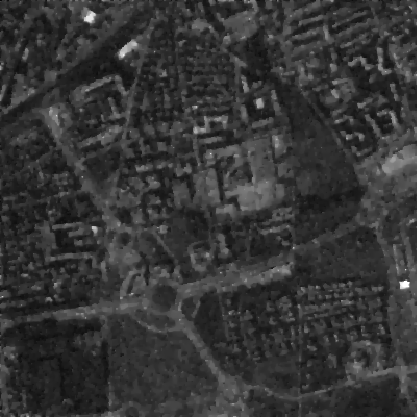}
	}
	\caption{Experimental results of Ours, AA, OS and DD for the Nimes with multiplicative Gamma noise $L=4$}\label{NimesL=4}
\end{figure}


\subsection*{Acknowledgment}
This work is partially supported by the National Natural Science Foundation of China (12301536, U21B2075, 12171123, 12371419, 12271130), the Natural Science Foundation of Heilongjiang Province
 (ZD2022A001), the Fundamental Research Funds for the Central Universities (HIT.NSRIF202302, 2022FRFK060020, 2022FRFK060031, 2022FRFK060014).

\section*{\small
 Conflict of interest} 

 {\small
 The authors declare that they have no conflict of interest.}


\begin{thebibliography}{99}
\bibitem{tur1982} 
Tur, M., Chin, K.C., Goodman, J.W.: \textit{When is speckle noise multiplicative?} Appl. Opt. \textbf{21}, 1157--1159 (1982) \href{https://doi.org/10.1364/AO.21.001157}{https://doi.org/10.1364/AO.21.001157}

\bibitem{hollig1983}
Hollig, K.: \textit{Existence of Infinitely Many Solutions for a Forward Backward Heat Equation.} Trans. Amer. Math. Soc. 278(1), 299–316 (1983) \href{https://doi.org/10.2307/1999317}{https://doi.org/10.2307/1999317}

\bibitem{lair1985}
Lair, A.V.: \textit{Uniqueness for a forward backward diffusion equation.} Trans. Amer. Math. Soc. 291(1), 311-317 (1985) \href{https://doi.org/10.1090/S0002-9947-1985-0797062-5}{https://doi.org/10.1090/S0002-9947-1985-0797062-5}

\bibitem{zhang2006} 
Zhang, K.: \textit{On existence of weak solutions for one-dimensional forward–backward diffusion equations.} J. Differential Equations \textbf{220}(2), 322--353 (2006) \href{https://doi.org/10.1016/j.jde.2005.01.011}{https://doi.org/10.1016/j.jde.2005.01.011}

\bibitem{elliott1985} 
Elliott, C.M.: \textit{The Stefan problem with a non-monotone constitutive relation.} IMA J. Appl. Math. \textbf{35}(2), 257--264 (1985) \href{https://doi.org/10.1093/imamat/35.2.257}{https://doi.org/10.1093/imamat/35.2.257}


\bibitem{ball1989} 
Ball, J.M.: \textit{A version of the fundamental theorem for young measures.} In: Rascle, M., Serre, D., Slemrod, M. (eds.) PDEs and Continuum Models of Phase Transitions (Nice, 1988). Lecture Notes in Physics, vol. 344, pp. 207--215. Springer, Heidelberg (1989) \href{https://doi.org/10.1007/BFb0024945}{https://doi.org/10.1007/BFb0024945}

\bibitem{ballmurat1989} 
Ball, J.M., Murat, F.: \textit{Remarks on Chacon’s biting lemma.} Proc. Amer. Math. Soc. \textbf{107}(3), 655--663 (1989) \href{https://doi.org/10.1090/S0002-9939-1989-0984807-3}{https://doi.org/10.1090/S0002-9939-1989-0984807-3}

\bibitem{slemrod1991} 
Slemrod, M.: \textit{Dynamics of measured valued solutions to a backward-forward heat equation.} J. Dynamics Differential Equations. \textbf{3}, 1--28 (1991) \href{https://doi.org/10.1007/BF01049487}{https://doi.org/10.1007/BF01049487}

\bibitem{young1969} 
Young, L.C.: \emph{Lectures on the Calculus of Variations and Optimal Control Theory.} W. B. Saunders Company, Philadelphia-London-Toronto, 1969 \href{https://mathscinet.ams.org/mathscinet/relay-station?mr=0259704}{https://mathscinet.ams.org/mathscinet/relay-station?mr=0259704}

\bibitem{kinderlehrer1992} 
Kinderlehrer, D., Pedregal, P.: \textit{Weak convergence of integrands and the Young measure representation.} SIAM J. Math. Anal. \textbf{23}(1), 1--19 (1992) \href{https://doi.org/10.1137/0523001}{https://doi.org/10.1137/0523001}

\bibitem{kinderlehrer1994} 
Kinderlehrer, D., Pedregal, P.: \textit{Gradient Young measures generated by sequences in Sobolev spaces.} J. Geom. Anal. \textbf{4}(1), 59--90 (1994) \href{https://doi.org/10.1007/BF02921593}{https://doi.org/10.1007/BF02921593}

\bibitem{diPerna1985} 
DiPerna, R.J.: \textit{Measure-valued solutions to conservation laws.} Arch. Rational Mech. Anal. \textbf{88}, 223--270 (1985) \href{https://doi.org/10.1007/BF00752112}{https://doi.org/10.1007/BF00752112}



\bibitem{li2010} 
Li, F., Ng, M.K., Shen, C.: \textit{Multiplicative noise removal with spatially varying regularization parameters.} SIAM J. Imaging Sci. \textbf{3}(1), 1--20 (2010) \href{https://doi.org/10.1137/090748421}{https://doi.org/10.1137/090748421}


\bibitem{brezis2010} 
Brezis, H.: \emph{Functional Analysis, Sobolev Spaces and Partial Differential Equations.} Universitext, Springer New York, 2010 \href{https://doi.org/10.1007/978-0-387-70914-7}{https://doi.org/10.1007/978-0-387-70914-7}


\bibitem{rudin1992} 
Rudin, L.I., Osher, S., Fatemi, E.: \textit{Nonlinear total variation based noise removal algorithms.} Physica D \textbf{60}(1-4), 259--268 (1992) \href{https://doi.org/10.1016/0167-2789(92)90242-F}{https://doi.org/10.1016/0167-2789(92)90242-F}

\bibitem{rudin2003mul} 
Rudin, L., Lions, P.L., Osher, S.: \textit{Multiplicative denoising and deblurring: Theory and algorithms.} In: Geometric Level Set Methods in Imaging, Vision, and Graphics, pp. 103--119. Springer, New York (2003) \href{https://doi.org/10.1007/0-387-21810-6_6}{https://doi.org/10.1007/0-387-21810-6\_6}

\bibitem{demoulini1996} 
Demoulini, S.: \textit{Young measure solutions for a nonlinear parabolic equation of forward-backward type.} SIAM J. Math. Anal. \textbf{27}(2), 376--403 (1996) \href{https://doi.org/10.1137/S0036141094261847}{https://doi.org/10.1137/S0036141094261847}

\bibitem{demoulini1996high} 
Demoulini, S.: \textit{Variational methods for Young measure solutions of nonlinear parabolic evolutions of forward-backward type and of high spatial order.} Appl. Anal. \textbf{63}(3-4), 363--373 (1996) \href{https://doi.org/10.1080/00036819608840514}{https://doi.org/10.1080/00036819608840514}

\bibitem{evans1990} 
Evans, L.C.: \textit{Weak convergence methods for nonlinear partial differential equations.} In: CBMS Regional Conference Series in Mathematics. Published for the Conference Board of the Mathematical Sciences, Washington, DC. American Mathematical Society, vol. 74,  Providence, RI, (1990) \href{https://mathscinet.ams.org/mathscinet/relay-station?mr=1034481}{https://mathscinet.ams.org/mathscinet/relay-station?mr=1034481}



\bibitem{yin2003} 
Yin, J., Wang, C.: \textit{Young measure solutions of a class of forward--backward diffusion equations.} J. Math. Anal. Appl. \textbf{279}(2), 659--683 (2003) \href{https://doi.org/10.1016/S0022-247X(03)00054-4}{https://doi.org/10.1016/S0022-247X(03)00054-4}

\bibitem{wang2014young} 
Wang, C., Nie, Y., Yin, J.: \textit{Young measure solutions for a class of forward-backward convection-diffusion equations.} Quart. Appl. Math. \textbf{72}(1), 177--192 (2014) \href{https://doi.org/10.1090/S0033-569X-2014-01338-8}{https://doi.org/10.1090/S0033-569X-2014-01338-8}

\bibitem{dong2013convex} 
Dong, Y., Zeng, T.: \textit{A convex variational model for restoring blurred images with multiplicative noise.} SIAM J. Imaging Sci. \textbf{6}(3), 1598--1625 (2013) \href{https://doi.org/10.1137/120870621}{https://doi.org/10.1137/120870621}


\bibitem{guo2011four} 
Guo, B., Gao, W.: \textit{Study of weak solutions for a fourth-order parabolic equation with variable exponent of nonlinearity.} Z. Angew. Math. Phys. \textbf{62}(5), 909--926 (2011) \href{https://doi.org/10.1007/s00033-011-0148-x}{https://doi.org/10.1007/s00033-011-0148-x}


\bibitem{evans2022partial} 
Evans, L.C.: \emph{Partial Differential Equations.} American Mathematical Society, vol. 19, (2022)

\bibitem{dacorogna2007} 
Dacorogna, B: \emph{Direct methods in the calculus of variations.} Applied Mathematical Sciences, Vol. 78, (2008) \href{https://doi.org/10.1007/978-0-387-55249-1}{https://doi.org/10.1007/978-0-387-55249-1}

\bibitem{aubert2008} 
Aubert, G., Aujol, J.F.: \textit{A variational approach to removing multiplicative noise.} SIAM J. Appl. Math. \textbf{68}, 925--946 (2008) \href{https://doi.org/10.1137/060671814}{https://doi.org/10.1137/060671814}

\bibitem{shi2008} 
Shi, J., Osher, S.: \textit{A nonlinear inverse scale space method for a convex multiplicative noise model.} SIAM J. Imaging Sci. \textbf{1}(3), 294--321 (2008) \href{https://doi.org/10.1137/070689954}{https://doi.org/10.1137/070689954}

\bibitem{zhang2020} 
Zhang, Y., Li, S., Guo, Z., Wu, B.: \textit{An adaptive total variational despeckling model based on gray level indicator frame.} Inverse Problems Imaging. \textbf{15}(6), 1421--1450 (2020) \href{https://doi.org/10.3934/ipi.2020068}{https://doi.org/10.3934/ipi.2020068}

\bibitem{chen2002} 
Chen, Y., Wunderli, T.: \textit{Adaptive total variation for image restoration in BV space.} J. Math. Anal. Appl. \textbf{272}, 117--137 (2002) \href{https://doi.org/10.1016/S0022-247X(02)00141-5}{https://doi.org/10.1016/S0022-247X(02)00141-5}

\bibitem{liu2013} 
Liu, Q., Li, X., Gao, T.: \textit{A nondivergence p-Laplace equation in a removing multiplicative noise model.} Nonlinear Anal. Real World Appl. \textbf{14}(5), 2046--2058 (2013) \href{https://doi.org/10.1016/j.nonrwa.2013.02.008}{https://doi.org/10.1016/j.nonrwa.2013.02.008}

\bibitem{guozc2011}
Guo, Z., Sun, J., Zhang, D., Wu, B.: \textit{Adaptive Perona-Malik model based on the variable exponent for image denoising.} IEEE Transactions on Image Processing. \textbf{21}(3), 958--967 (2011) \href{https://doi.org/10.1109/TIP.2011.2169272}{https://doi.org/10.1109/TIP.2011.2169272}


\bibitem{dong2013} 
Dong, G., Guo, Z., Wu, B.: \textit{A convex adaptive total variation model based on the gray level indicator for multiplicative noise removal.} Abstr. Appl. Anal. \textbf{2013}, 1--21 (2013) \href{https://doi.org/10.1155/2013/912373}{ https://doi.org/10.1155/2013/912373}

\bibitem{li2024} 
Li, Y., Guo, Z., Shao, J., Li, Y., Wu, B.: \textit{Variable-order fractional 1-Laplacian diffusion equations for multiplicative noise removal.} Fract. Calculus Appl. Anal. 1--40 (2024) \href{https://doi.org/10.1007/s13540-024-00345-6}{https://doi.org/10.1007/s13540-024-00345-6}

\bibitem{guidotti2012backward} 
Guidotti, P.: \textit{A backward--forward regularization of the Perona--Malik equation.} J. Differential Equations. \textbf{252}(4), 3226--3244 (2012) \href{https://doi.org/10.1016/j.jde.2011.10.022}{https://doi.org/10.1016/j.jde.2011.10.022}


\bibitem{guidotti2013} 
Guidotti, P., Kim, Y., Lambers, J.: \textit{Image restoration with a new class of forward-backward-forward diffusion equations of Perona--Malik type with applications to satellite image enhancement.} SIAM J. Imaging Sci. \textbf{6}(3), 1416--1444 (2013) \href{https://doi.org/10.1137/120882895}{https://doi.org/10.1137/120882895}

\bibitem{zhou2014} 
Zhou, Z., Guo, Z., Dong, G., Sun, J., Zhang, D., Wu, B.: \textit{A doubly degenerate diffusion model based on the gray level indicator for multiplicative noise removal.} IEEE Trans. Image Process. \textbf{24}(1), 249--260 (2015) \href{https://doi.org/10.1109/TIP.2014.2376185}{https://doi.org/10.1109/TIP.2014.2376185}

\bibitem{perona1990} 
Perona, P., Malik, J.: \textit{Scale-space and edge detection using anisotropic diffusion.} IEEE Trans. Pattern Anal. Mach. Intell. \textbf{12}(7), 629--639 (1990) \href{https://doi.org/10.1109/34.56205}{https://doi.org/10.1109/34.56205}

\bibitem{yaowj2019}
Yao, W., Guo, Z., Sun, J., Wu, B., Gao, H.: \textit{Multiplicative noise removal for texture images based on adaptive anisotropic fractional diffusion equations.} SIAM Journal on Imaging Sciences. \textbf{12}(2), 839--873 (2019) \href{https://doi.org/10.1137/18M1187192}{https://doi.org/10.1137/18M1187192}

\bibitem{shao2020} 
Shao, J., Guo, Z., Shan, X., Zhang, C., Wu, B.: \textit{A new non-divergence diffusion equation with variable exponent for multiplicative noise removal.} Nonlinear Anal. Real World Appl. \textbf{56}, 103166 (2020) \href{https://doi.org/10.1016/j.nonrwa.2020.103166}{https://doi.org/10.1016/j.nonrwa.2020.103166}

\bibitem{gaotl2022}
Gao, T., Liu, Q., Zhang, Z.: \textit{Fractional 1-Laplacian evolution equations to remove multiplicative noise.} Discrete and Continuous Dynamical Systems - B. \textbf{27}(9), 4837--4854 (2022) \href{https://doi.org/10.3934/dcdsb.2021254}{https://doi.org/10.3934/dcdsb.2021254}

\bibitem{shao2022} 
Shao, J., Guo, Z., Yao, W., Yan, D., Wu, B.: \textit{A Non-Local Diffusion Equation for Noise Removal.} Acta Math. Sci. \textbf{42}, 1779--1808 (2022) \href{https://doi.org/10.1007/s10473-022-0505-1}{https://doi.org/10.1007/s10473-022-0505-1}

\bibitem{durand2010} 
Durand, S., Fadili, J., Nikolova, M.: \textit{Multiplicative noise removal using L1 fidelity on frame coefficients.} J. Math. Imaging Vision. \textbf{36}, 201--226 (2010) \href{https://doi.org/10.1007/s10851-009-0180-z}{https://doi.org/10.1007/s10851-009-0180-z}

\bibitem{visintin2002} 
Visintin, A.: \textit{Forward–backward parabolic equations and hysteresis.} Calc. Var. Partial Differential Equations \textbf{15}(1), 115--132 (2002) \href{https://doi.org/10.1007/s005260100120}{https://doi.org/10.1007/s005260100120}

\bibitem{guidotti2009} 
Guidotti, P., Lambers, J.V.: \textit{Two new nonlinear nonlocal diffusions for noise reduction.} J. Math. Imaging Vision \textbf{33}(1), 25--37 (2009) \href{https://doi.org/10.1007/s10851-008-0108-z}{https://doi.org/10.1007/s10851-008-0108-z}

\bibitem{acerbi1984} 
Acerbi, E., Fusco, N.: \textit{Semicontinuity problems in the calculus of variations.} Arch. Rational Mech. Anal. \textbf{86}(2), 125--145 (1984) \href{https://doi.org/10.1007/BF00275731}{ https://doi.org/10.1007/BF00275731}


\end{thebibliography}
\end{document}